\documentclass{article}
\usepackage{amsfonts,amsmath,amssymb,graphics,amsthm}
\usepackage[english]{babel}
 \usepackage[all]{xy}
 \usepackage[latin1]{inputenc}
\usepackage{lmodern}
\usepackage{tikz}
\usepackage{xypic}
\xyoption{all}
\input{xypic}
\newdir{ >}{{}*!/-8pt/\dir{>}}
\def\R{\mathbb{R}}

\def\C{\mathbb{C}}

\def\K{\mathbb{K}} 
\def\pr{$\bf{Proof.}$\quad}
\def\fin{\hfill$\square$\\}
   \def\im{{\rm im}\,}   \def\ker{{\rm ker}\,}
     \def\Hom{{\rm Hom}\,}   \def\End{{\rm End}\,}
     \def\der{{\rm der}\,}   \def\coker{{\rm coker}\,}

\newtheorem{theo}{Theorem}
\newtheorem{defi}{Definition}
\newtheorem{ex}{Example}
\newtheorem{rem}{Remark}
\newtheorem{prop}{Proposition}
\newtheorem{cor}{Corollary}
\newtheorem{lem}{Lemma}
\begin{document}
\title{On the string Lie algebra} 

\author{Salim Rivi\`ere\\
        Universit\'e du Luxembourg\\
        salim.riviere@uni.lu\\
\and    Friedrich Wagemann \\
        Universit\'e de Nantes\\
        wagemann@math.univ-nantes.fr}

\maketitle

\begin{abstract}
We construct an abelian representative for the crossed module associated 
to the string Lie algebra. We show how to apply this construction in order 
to define quasi-invariant tensors which serve to categorify the infinitesimal
braiding on the category of ${\mathfrak g}$-modules given by an r-matrix, 
following Cirio-Martins \cite{CirMar}.  
\end{abstract}

\section*{Introduction}

Given a simple compact simply connected Lie group $G$, the string group ${\rm Str}(G)$ is 
in some sense its 3-connected cover. There are various models for ${\rm Str}(G)$ in the 
literature, for example as a weak 2-group with finite-dimensional components or as
a strict 2-group with infinite-dimensional components \cite{BCSS}. But strict 2-groups
are nothing else than crossed modules, thus the string group and its Lie algebra, the 
{\it string Lie algebra}, can be represented by crossed modules, see for example
\cite{Nee} for some model. Crossed modules of Lie algebras are closely related to 
3-cohomology classes, and the 3-cohomology class corresponding to the string Lie 
algebra is the class represented by the {\it Cartan cocycle} 
$\kappa([-,-],)\in Z^3({\mathfrak g},\C)$ for a simple complex Lie algebra ${\mathfrak g}$
and its symmetric, non-degenerate bilinear Killing form $\kappa$. 
Introductive material on crossed modules is given in Section 1, and material on the string 
Lie algebra in Section 3 (closely following \cite{Nee}). 

We want to emphasize in this article that it is interesting to find {\it abelian 
representatives} for a given equivalence class of crossed modules. These are crossed 
modules of the form $\mu:V_2\to V_3\times_{\alpha}{\mathfrak g}$ which are spliced together
from a short exact sequence of ${\mathfrak g}$-modules 
$$0\to V_1\to V_2\stackrel{d}{\to} V_3\to 0$$
and an abelian extension $V_3\times_{\alpha}{\mathfrak g}$ of ${\mathfrak g}$ by $V_3$
via a 2-cocycle $\alpha$. We will almost always consider crossed modules where $V_1=\C$,
the trivial ${\mathfrak g}$-module. The construction of abelian representatives stems 
back to \cite{Wag} and is recalled here in Section 2.  

Unfortunately, the explicit expression for the 2-cocycle $\alpha$ which the second 
author presented in 
\cite{Wag} in order to construct an abelian representative for the Cartan cocycle 
corresponding to a general simple complex Lie algebra is wrong. We thank Lucio Cirio 
for this observation. In the present paper, we propose a new expression for $\alpha$
(see Equation \ref{eq:formule-definition-alpha}) which is kind of abstract, but
which works for any constant valued 3-cocycle. The formula relies on an explicit formula 
for the homotopy contracting the Koszul complex and appeared first (to our knowledge) 
in the first author's doctoral thesis. 
In order to render its expression 
explicit, one has to compute the Eulerian idempotent on the corresponding universal
enveloping algebra, and this is in general a difficult problem. We illustrate this 
at the end of Section 4. 

Cirio and Martins applied the abelian representative of \cite{Wag} for 
${\mathfrak s}{\mathfrak l}_2(\C)$ which is given in terms of vector fields 
and formal differential forms to the problem of categorifying 
the infinitesimal braiding in the category of 
${\mathfrak g}$-modules which is associated to an r-matrix for ${\mathfrak g}$, i.e. a 
symmetric ${\mathfrak g}$-invariant tensor $r\in{\mathfrak g}\otimes{\mathfrak g}$.
They showed that {\it quasi-invariant tensors} (see Definition \ref{quasiinvariant_tensor})
lead to such a braiding. 
As an application of our theory, we show in Section 5 how to construct quasi-invariant tensors 
for a wide class of Lie algebras, including all simple complex Lie algebra, and therefore
extending Cirio and Martins' construction.

For this, we had to formalize one compatibility condition which was automatically satisfied 
in Cirio-Martins' framework. Let $\mu:V_2\to V_3\times_{\alpha}{\mathfrak g}$ be an abelian
representative for a crossed module corresponding to a trivial-valued 3-cocycle $\gamma$. 
Take a linear section $Q$ of the map $d:V_2\to V_3$, and lift $\alpha$ to 
$w:=Q\circ\alpha$. Take its Chevalley-Eilenberg coboundary $d_{CE}w=:\Phi$. 
On the other hand, consider a ${\mathfrak g}$-invariant symmetric 2-tensor 
$r=\sum_is_i\otimes t_i$. 
Then the condition reads for all $X,Y\in{\mathfrak g}$:
 $$\sum_i\Phi(s_i,X,Y)\otimes t_i\,=\,1_{V_2}\otimes[X,Y],\,\,\,\,{\rm and}\,\,\,\,     
\sum_is_i\otimes\Phi(t_i,X,Y)\,=\,[X,Y]\otimes 1_{V_2},$$
where $1_{V_2}$ generates the image of $V_1=\C$ in $V_2$. 

\begin{theo}
Suppose that in the above situation the above compatibility condition is satisfied. 

Then the triple $(\overline{r},\xi,c)$ (defined in Section \ref{construction}) 
is a quasi-invariant tensor for the crossed module 
$\mu:V_2\to V_3\times_{\alpha}{\mathfrak g}$.
\end{theo}  

We show large classes of examples where this condition is satisfied, in particular 
it is satisfied for all (finite-dimensional) simple complex Lie algebras. 

\vspace{1cm}

\noindent{\bf Acknowledgements:}

FW thanks Lucio Cirio for various useful discussion on the matter of this paper.

\section{Crossed modules}

Fix a field $\K$. Mainly we will consider $\K=\C$. 

\begin{defi}   \label{def_crmod}
A crossed module of Lie algebras is a homomorphism of
Lie algebras $\mu:\mathfrak{m}\to\mathfrak{n}$ together with 
an action of $\mathfrak{n}$ on $\mathfrak{m}$ by derivations, 
denoted $m\mapsto n\cdot m$, such that

(a) $\mu(n\cdot m)\,=\,[n,\mu(m)]$ for all $n\in\mathfrak{n}$
and all $m\in\mathfrak{m}$,

(b) $\mu(m)\cdot m'\,=\,[m,m']$ for all $m,m'\in\mathfrak{m}$.
\end{defi}

\begin{rem}    \label{exact_sequence}
To each crossed module of Lie algebras $\mu:\mathfrak{m}\to\mathfrak{n}$, 
one associates a four term exact sequence
\begin{displaymath}
0\to
V\stackrel{i}{\to}\mathfrak{m}\stackrel{\mu}{\to}
\mathfrak{n}\stackrel{\pi}{\to}\mathfrak{g}\to
0
\end{displaymath}
where $\ker(\mu)\,=:\,V$ and $\mathfrak{g}\,:=\,\coker(\mu)$.
\end{rem}

\begin{rem}
\begin{itemize}
\item By (a), $\mathfrak{g}$ is a Lie algebra, 
because $\im\,(\mu)$ is an ideal. 
\item By (b), $V$ is a central Lie subalgebra of 
$\mathfrak{m}$, and in particular abelian.
\item By (a), the action of $\mathfrak{n}$ on $\mathfrak{m}$ 
induces a structure of a $\mathfrak{g}$-module on $V$. 
\item Note that in general
$\mathfrak{m}$ and $\mathfrak{n}$ are not $\mathfrak{g}$-modules.
\end{itemize}
\end{rem}
 
Let us recall how to associate to a crossed module a $3$-cocycle of 
$\mathfrak{g}$ with values in $V$. Recall the exact sequence from Remark \ref{exact_sequence}:
\begin{displaymath}
0\to
V\stackrel{i}{\to}\mathfrak{m}\stackrel{\mu}{\to}
\mathfrak{n}\stackrel{\pi}{\to}\mathfrak{g}\to
0
\end{displaymath}
The first step is to take a linear section $\rho$ of $\pi$ and
to compute the default of $\rho$ to be a Lie algebra homomorphism,
i.e.
\begin{displaymath}
\alpha(x_1,x_2):=[\rho(x_1),\rho(x_2)]-\rho([x_1,x_2]).
\end{displaymath}
Here, $x_1,x_2\in\mathfrak{g}$. $\alpha$ is bilinear and skewsymmetric 
in $x_1,x_2$. We have obviously $\pi(\alpha(x_1,x_2))=0$, because $\pi$ is a Lie
algebra homomorphism, so $\alpha(x_1,x_2)\in \im(\mu)=\ker(\pi)$. By exactness,
there exists $\beta(x_1,x_2)\in\mathfrak{m}$ such that
\begin{displaymath}
\mu(\beta(x_1,x_2))\,=\,\alpha(x_1,x_2).
\end{displaymath}
Choosing a linear section $\sigma$ on $\im(\mu)$, one can choose $\beta$ as 
\begin{displaymath}
\beta(x_1,x_2)\,=\,\sigma(\alpha(x_1,x_2))
\end{displaymath}
showing that we can suppose $\beta$ bilinear and skewsymmetric in $x_1,x_2$.

Now, one computes $\mu(d^{\mathfrak{m}}\beta(x_1,x_2,x_3))=0$, where 
$d^{\mathfrak{m}}$ is the formal expression of the Lie algebra coboundary operator, 
while the map $\eta\circ\rho$ is not an action in general.

Choosing a linear section $\tau$ on $i(V)=\ker(\mu)$, one can choose 
$\gamma$ to be 
\begin{equation}  \label{def_3_cocycle}  
\gamma\,:=\,\tau\circ d^{\mathfrak{m}}\beta
\end{equation} 
gaining that $\gamma$ is 
also trilinear and skewsymmetric in $x_1,x_2$ and $x_3$. One computes 
that $\gamma$ is a $3$-cocycle of $\mathfrak{g}$ with values in $V$. 

On the other hand, one may introduce an equivalence 
relation on the class of crossed modules and an abelian group structure. 
With this structure at hand, the above construction of the cocycle $\gamma$
leads to the following theorem.  

\begin{theo}[Gerstenhaber]
There is an isomorphism of abelian groups
\begin{displaymath}
b:{\rm crmod}(\mathfrak{g},V)\,\cong\,H^3(\mathfrak{g},V).
\end{displaymath}
\end{theo}

A proof of this theorem is (more or less) contained in \cite{Ger64} and \cite{Ger66}.
According to \cite{ML}, we attribute it to Murray Gerstenhaber. 
We refer to \cite{Wag} for further computational details on this theorem. 

The main question we want to address here is how to construct crossed modules which
relate to a given $3$-cocycle under the correspondence of this theorem. 

\section{Construction of crossed modules}

\subsection{Construction using a short exact sequence of modules}

Let us ask the following question: Given a Lie algebra $\mathfrak{g}$, a short exact 
sequence of $\mathfrak{g}$-modules
\begin{equation}    \label{*}
0\to V_1\to V_2\to V_3\to 0
\end{equation}
(regarded as a short exact sequence of abelian Lie algebras) and an abelian extension 
${\mathfrak r}$ of $\mathfrak{g}$ by the abelian Lie algebra $V_3$
\begin{equation}   \label{**}
0\to V_3\to \mathfrak{r}\to \mathfrak{g}\to 0,
\end{equation}
is the Yoneda product of (\ref{*}) and (\ref{**}) a crossed module ? 
In case (\ref{**}) is given by a $2$-cocycle $\alpha$, we use the notation 
$\mathfrak{r}=V_3\times_{\alpha}{\mathfrak g}$.

\begin{theo}  \label{construction_thm}
In the above situation, the Yoneda product of $(\ref{*})$ and $(\ref{**})$, namely
$$0\to V_1\to V_2\to \mathfrak{r}=V_3\times_{\alpha}{\mathfrak g}\to \mathfrak{g}\to 0,$$
is a crossed module, 
the associated $3$-cocycle of which is (cohomologuous to) the image 
$\partial\alpha$ of $\alpha$,
i.e. the $2$-cocycle defining the central extension $(\ref{**})$, 
under the connecting homomorphism 
$\partial$ in the long exact cohomology sequence associated to $(\ref{*})$.
\end{theo}

This theorem is proved in \cite{Wag}. It is our main device to construct crossed modules
for a given $3$-cohomology class. In fact, any equivalence class of crossed modules 
contains a representative of the above form:

\begin{prop}
Let $\mu:{\mathfrak m}\to{\mathfrak n}$ be a crossed module of Lie algebras. Then there 
exists a crossed module of the above constructed type 
$$\mu_1:V_2\to \mathfrak{r}=V_3\times_{\alpha}{\mathfrak g}$$
which is equivalent to $\mu:{\mathfrak m}\to{\mathfrak n}$.
\end{prop}

\pr To $\mu:{\mathfrak m}\to{\mathfrak n}$ belongs an exact sequence
\begin{displaymath}
0\to
V\stackrel{i}{\to}\mathfrak{m}\stackrel{\mu}{\to}
\mathfrak{n}\stackrel{\pi}{\to}\mathfrak{g}\to
0
\end{displaymath}
with $\ker(\mu)\,=:\,V$ and $\mathfrak{g}\,:=\,\coker(\mu)$, and a cohomology class 
$[\gamma]\in H^3(\mathfrak{g},V)$ by Gerstenhaber's theorem. Now, embed the 
$\mathfrak{g}$-module $V=V_1$ into an injective $\mathfrak{g}$-module $V_2$, 
and form the quotient $\mathfrak{g}$-module $V_3$ yielding an exact sequence
$$0\to V_1\to V_2\to V_3\to 0.$$
As $V_2$ is injective, the connecting homomorphism
$$\partial:H^2(\mathfrak{g},V_3)\to H^3(\mathfrak{g},V_1)$$
is an isomorphism, thus the class $[\gamma]$ has a preimage $[\alpha]\in 
H^2(\mathfrak{g},V_3)$. By Theorem \ref{construction_thm}, the data defines
a crossed module of the desired type.\fin 

Having seen that we can always find {\it some} $2$-cocycle $\alpha$ such that 
the above glued crossed module represents a given cohomology class, 
the only difficulty in this approach is to find a 
{\it meaningful} $2$-cocycle $\alpha$ associated to a given crossed module.

\begin{defi}
Given an equivalence class of crossed modules $[\mu:{\mathfrak m}\to{\mathfrak n}]
\in{\rm crmod}({\mathfrak g},V)$ (or, equivalently, a class 
$[\gamma]\in H^3({\mathfrak g},V)$), a crossed module of the form 
$$\mu':V_2\to V_3\times_{\alpha}{\mathfrak g}$$
for some ${\mathfrak g}$-modules $V_2$ and $V_3$ and some 2-cocycle 
$\alpha\in Z^2({\mathfrak g},V_3)$ which is equivalent to $\mu$  is called 
an abelian representative of the class $[\mu:{\mathfrak m}\to{\mathfrak n}]$.
\end{defi}   

\subsection{Construction from a central extension of an ideal}

Following Neeb \cite{Nee}, we look in this section at a crossed module of 
Lie algebras as the data of a general extension of Lie algebras
\begin{equation} \label{general_extension}
0\to {\mathfrak l}\stackrel{\iota}{\to}\widehat{\mathfrak g}\stackrel{\pi}{\to} 
{\mathfrak g}\to 0,
\end{equation}
together with a central extension 
$$0\to {\mathfrak z}\to\widehat{\mathfrak l}\stackrel{p}{\to}{\mathfrak l}
\to 0$$
of the ideal ${\mathfrak l}$ of the extension (\ref{general_extension}).

If we want that the given general extension and central extension of 
${\mathfrak l}$ form a crossed module 
$$\mu:=\widehat{\mathfrak l}\stackrel{p}{\to}{\mathfrak l}\stackrel{\iota}{\to}
\widehat{\mathfrak g},$$
they have to be compatible in the sense that the action of ${\mathfrak l}$ on 
$\widehat{\mathfrak l}$ extends 
to an action of $\widehat{\mathfrak g}$ on $\widehat{\mathfrak l}$ satisfying 
the requirements of Definition \ref{def_crmod}.

In the following, we will use this slightly different framework to derive 
once again the cohomology class $[\gamma]$ associated to a crossed module,
i.e. in the following, we assume the data of the two extensions together with
the extension of the action to form a crossed module. In particular, there 
is a fixed action of ${\mathfrak g}$ on ${\mathfrak z}$.    

The Lie bracket on the central extension $\widehat{\mathfrak l}$ can be given 
in a very explicit way. Namely, we may assume 
$\widehat{\mathfrak l}={\mathfrak z}\oplus_{\omega}{\mathfrak l}$, meaning
that the bracket reads as
$$[(z_1,l_1),(z_2,l_2)]=(\omega(l_1,l_2),[l_1,l_2]).$$ 
Here $\omega\in Z^2({\mathfrak l},{\mathfrak z})$ is the $2$-cocycle defining
the central extension. Note that ${\mathfrak l}$ acts trivially on 
${\mathfrak z}$, and the given ${\mathfrak g}$-action on ${\mathfrak z}$ 
extends trivially to a $\widehat{\mathfrak g}$-action, still denoted
$z\mapsto x\cdot z$.  

Now this $\widehat{\mathfrak g}$-action on ${\mathfrak z}$ is supposed to
extend to an action on $\widehat{\mathfrak l}$, denoted by
\begin{equation}  \label{definition_of_theta}
x\cdot (z,l)= (x\cdot z + \theta(x,l),[x,l])
\end{equation}
with some linear map $\theta:\widehat{\mathfrak g}\oplus{\mathfrak l}\to
{\mathfrak z}$. Property {\it (b)} of Definition \ref{def_crmod} implies
for $l_1\in{\mathfrak l}\subset\widehat{\mathfrak g}$
$$l_1\cdot (z,l_2)=[(0,l_1),(z,l_2)]=(\omega(l_1,l_2),[l_1,l_2]),$$
meaning that $\theta|_{{\mathfrak l}\times{\mathfrak l}}=\omega$. 

\begin{lem} \label{lemma_conditions_on_theta}
Fix an element $x\in\widehat{\mathfrak g}$.
\begin{enumerate}
\item[(a)] A linear map $\theta_x\in\Hom({\mathfrak l},{\mathfrak z})$ defines a
derivation
$$\zeta(x):\widehat{\mathfrak l}\to\widehat{\mathfrak l},\quad(z,l)
\mapsto(x\cdot z+\theta_x(l),[x,l])$$
if and only if $d^{\mathfrak l}\theta_x=x\cdot\omega$, where $d^{\mathfrak l}$ is
the Lie algebra coboundary operator of ${\mathfrak l}$ with values in 
$\Hom({\mathfrak l},{\mathfrak z})$ and $\omega$ is seen as
$\omega\in C^1({\mathfrak l},\Hom({\mathfrak l},{\mathfrak z}))$. 

The condition 
$d^{\mathfrak l}\theta_x=x\cdot\omega$ means explicitely for all 
$l,l'\in{\mathfrak l}$
$$x\cdot\omega(l,l')-\omega([x,l],l')-\omega([l',x],l)+\theta_x([l,l'])=0.$$
\item[(b)] Suppose that the linear map  $\theta:\widehat{\mathfrak g}\to
\Hom({\mathfrak l},{\mathfrak z}),x\mapsto\theta_x$ satisfies condition (a).

Then $\theta$ defines a representation of $\widehat{\mathfrak g}$ on 
$\widehat{\mathfrak l}$ via the formula
$$x\cdot(z,l)=(x\cdot z+\theta_x(l),[x,l])$$
if and only if $\theta$ is a $1$-cocycle w.r.t. the natural 
$\widehat{\mathfrak g}$-structure on 
$C^1({\mathfrak l},{\mathfrak z})=\Hom({\mathfrak l},{\mathfrak z})$.

This condition means explicitely for $x,x'\in\widehat{\mathfrak g}$ and 
$l\in{\mathfrak l}$
$$x\cdot\theta(x',l)-x'\cdot\theta(x,l)-\theta([x,x'],l)+\theta(x,[x',l])+
\theta(x',[l,x])=0.$$
\end{enumerate}
\end{lem} 

\pr see \cite{Nee}.\fin

We saw that the action of $\widehat{\mathfrak g}$ on $\widehat{\mathfrak l}$
is determined by the map $\theta:\widehat{\mathfrak g}\times{\mathfrak l}\to
{\mathfrak z}$. The preceding Lemma states conditions on $\theta$ coming from 
the fact that $\widehat{\mathfrak g}\to\End(\widehat{\mathfrak l})$ is a 
representation (condition {\it (b)}) and that it is a representations by 
derivations, i.e. in fact $\widehat{\mathfrak g}\to\der(\widehat{\mathfrak l})$
(condition {\it (a)}).

Now we define the $3$-cocycle associated in this context to the crossed 
modules we started with. $\theta|_{{\mathfrak l}\times{\mathfrak l}}$ is alternating,
thus we can extend $\theta:\widehat{\mathfrak g}\times{\mathfrak l}\to
{\mathfrak z}$ by skewsymmetry to a bilinear map 
\begin{equation}   \label{definition_theta_tilde}
\widetilde{\theta}:\widehat{\mathfrak g}\times\widehat{\mathfrak g}\to
{\mathfrak z}.
\end{equation}
Then $d^{\widehat{\mathfrak g}}\widetilde{\theta}\in Z^3(\widehat{\mathfrak g},
{\mathfrak z})$ is a $3$-cocycle which vanishes on 
${\mathfrak l}\times\widehat{\mathfrak g}^2$ by condition {\it (b)} of the Lemma.
This means that as soon as the cocycle is evaluated on a triple of elements 
where at least one comes from ${\mathfrak l}$, it vanishes.
Therefore, it passes to the quotient ${\mathfrak g}=\widehat{\mathfrak g}/
{\mathfrak l}$ to a $3$-cocycle $\chi\in Z^3({\mathfrak g},{\mathfrak z})$
with $d^{\widehat{\mathfrak g}}\widetilde{\theta}=\pi^*\chi$.

\begin{rem}
Observe that one advantage of this approach is that here the coboundary 
operator $d^{\widehat{\mathfrak g}}$ refers to a genuine action of 
$\widehat{\mathfrak g}$ on ${\mathfrak z}$, and we need not deal with 
formal expressions of coboundary operators.
\end{rem}
   
\begin{lem} \label{def_3_cocycle_bis}
The $3$-cocycle $\chi$ defines a class $[\chi]\in H^3({\mathfrak g},
{\mathfrak z})$ which does not depend on the choice of the cocycle $\omega$,
neither on the choice of the extension $\widetilde{\theta}$ of $\theta$.    
\end{lem}

\pr This is shown in \cite{Nee}.\fin

Let us come back to the formula for the characteristic class of a crossed
module which we saw in the previous section. 

\begin{prop} \label{comparison_prop}
Let $\mu:{\mathfrak m}\to{\mathfrak n}$ be a crossed module. Let
$$0\to V\stackrel{i}{\to} {\mathfrak m}\stackrel{\mu}{\to}{\mathfrak n}
\stackrel{\pi}{\to}{\mathfrak g}\to 0$$
be the associated four term exact sequence and denote  
${\mathfrak l}=\im(\mu)$.

Then the classes $[\gamma]$, where $\gamma\in Z^3({\mathfrak g},V)$ 
is obtained as
in Equation (\ref{def_3_cocycle}) by using sections in the four term exact sequence, and 
$[\chi]$, where $\chi=d^{\hat{\mathfrak g}}\tilde{\theta}$ is defined in 
Lemma \ref{def_3_cocycle_bis} (with ${\mathfrak z}=V$), coincide in 
$H^3({\mathfrak g},V)$.
\end{prop}

\pr
First choose $\omega\in Z^2({\mathfrak l},V)$ with 
${\mathfrak m}=:\hat{\mathfrak l}=
V\oplus_{\omega}{\mathfrak l}$ and a linear section $\rho:{\mathfrak g}\to
{\mathfrak n}$. Associated to $\rho$, we have for $x,y,z\in{\mathfrak g}$
$$\alpha(x,y)=[\rho(x),\rho(y)]-\rho([x,y])\in{\mathfrak n},$$
and $\beta(x,y)\in{\mathfrak m}$ with
$$\mu\beta(x,y)=\alpha(x,y).$$
We therefore get $d^{\mathfrak m}\beta(x,y,z)$ and $\gamma(x,y,z)\in V$ such that
$i\gamma(x,y,z)=d^{\mathfrak m}\beta(x,y,z)$ as before. Now as 
$\beta(x,y)\in\hat{\mathfrak l}=V\oplus_{\omega}{\mathfrak l}$, write 
$$\beta=(\beta_V,\alpha)\quad{\rm with}\quad \beta_V\in C^2({\mathfrak g},V).$$
On the other hand, write the action of ${\mathfrak m}$ on $\hat{\mathfrak l}$ as
$x\cdot(z,l)=(x\cdot z+\theta(x,l),[x,l])$ with a bilinear map 
$\theta:{\mathfrak n}\times{\mathfrak l}\to V$. Then
\begin{eqnarray*}
d^{\mathfrak m}\beta(x,y,z)&=&\sum_{\rm cycl.}\Big(\rho(x)\cdot(\beta_V(y,z),
\alpha(y,z))-(\beta_V([x,y],z),\alpha([x,y],z))\Big) \\
&=&\sum_{\rm cycl.}\Big((x\cdot\beta_V(y,z)+\theta(\rho(x),\alpha(y,z)),
[\rho(x),\alpha(y,z)])\\
&-&(\beta_V([x,y],z),\alpha([x,y],z))\Big). 
\end{eqnarray*}
As $d^{\mathfrak m}\beta(x,y,z)\in\ker(\mu)$, the ${\mathfrak l}$-component of 
this expression must be zero. This may be seen as an abstract Bianchi identity,
cf \cite{Nee}. We thus obtain
$$d^{\mathfrak m}\beta(x,y,z)=d^{V}\beta_V(x,y,z)+\sum_{\rm cycl.}
\theta(\rho(x),\alpha(y,z))=\gamma(x,y,z).$$  
We must now compare this to $\chi(x,y,z)$. For this, let $\tilde{\theta}\in
C^2({\mathfrak m},V)$ be an alternating extension of $\theta$. We may assume
that $\rho^*\tilde{\theta}=\beta_V$ up to a coboundary. We show now that 
$\pi^*(d^{\mathfrak m}\beta)=d^{\mathfrak m}\tilde{\theta}$, so that $[\chi]= 
[d^{\mathfrak m}\beta]\in H^3({\mathfrak g},V)$ as claimed. 
\begin{eqnarray*}
d^{\mathfrak m}\tilde{\theta}(x,y,z)&=&\sum_{\rm cycl.}\Big(x\cdot
\tilde{\theta}(\rho(y),\rho(z))-\tilde{\theta}([\rho(x),\rho(y)],\rho(z))\Big)\\
&=&\sum_{\rm cycl.}\Big(x\cdot\tilde{\theta}(\rho(y),\rho(z))-\tilde{\theta}
(\rho([x,y])+\alpha(x,y),\rho(z))\Big) \\
&=&\sum_{\rm cycl.}\Big(x\cdot\beta_V(y,z)-\beta_V([x,y],z)-\tilde{\theta}
(\alpha(x,y),\rho(z))\Big) \\
&=&\sum_{\rm cycl.}\Big(x\cdot\beta_V(y,z)-\beta_V([x,y],z)+\theta(\rho(x),
\alpha(y,z))\Big) \\
&=&d^{\mathfrak m}\beta_V(x,y,z)+\sum_{\rm cycl.}\theta(\rho(x),\alpha(y,z))\\
&=&d^{\mathfrak m}\beta(x,y,z),
\end{eqnarray*}
where we used the definition of $\alpha$ passing from the first to the second 
line, then the assumption $\rho^*\tilde{\theta}=\beta_V$, then the fact that
$\tilde{\theta}$ reduces to $\theta$ in case one variable is in 
${\mathfrak l}$, and finally the relation between $d^{\mathfrak m}\beta$ and 
$d^{\mathfrak m}\beta_V$ we obtained earlier. 
\fin

\section{The string Lie algebra}

The {\it string group} arises in the following context. Recall that for a
connected, simply connected, finite dimensional, semi-simple Lie group $G$, 
$\pi_3(G)$ counts the number of simple factors of $G$. Starting with an 
arbitrary topological group, the connected component of $1\in K$, denoted
$K_1$, has $\pi_i(K)=\pi_i(K_1)$ for all $i>0$, but $\pi_0(K_1)=0$. 
In the same way (in case $K$ is, e.g. locally 
contractible and connected), the universal covering $\tilde{K}$ of $K$, 
satisfies $\pi_i(\tilde{K})=\pi_i(K)$ for all $i>1$, but $\pi_1(\tilde{K})=
\pi_0(\tilde{K})=0$. All finite dimensional Lie groups $K$ have $\pi_2(K)=0$. 
The string group ${\rm Str}(G)$ of $G$ is defined in the same way as being a 
group $S$ such that $\pi_i(S)=\pi_i(G)$ for all $i>3$, but $\pi_3(S)=\pi_2(S)=
\pi_1(S)=\pi_0(S)=0$. By what we remarked before, the group $S$ cannot be a 
finite dimensional Lie group, because its maximal compact subgroup (which is 
homotopy equivalent to $S$) should contain a simple factor. The string group
${\rm Str}(G)$ for a connected, simply connected, finite dimensional, semi-simple 
Lie group $G$ has many homotopy theoretical realizations. We will focuss here
on the Lie algebraic version given in \cite{Nee} of it. 

Let ${\mathfrak g}$ be a finite dimensional real Lie algebra and denote by
$I$ the closed interval $I=[0,1]$. We consider 
the {\it smooth path algebra} 
$$P({\mathfrak g}):=\{\xi\in C^{\infty}(I,{\mathfrak g})\,|\,\xi(0)=0\}$$
of ${\mathfrak g}$. 
Then evaluation in $1\in I$,
denoted ${\rm ev}_1$, leads to a short exact sequence 
$$0\to{\mathfrak l}\to P({\mathfrak g})\to{\mathfrak g}\to 0,$$
where ${\mathfrak l}:=\ker({\rm ev}_1)$ is the ideal of closed paths in 
$P({\mathfrak g})$  and a linear section $\rho:{\mathfrak g}\to 
P({\mathfrak g})$ is given by $\rho(x)(t):=tx$. Note that ${\mathfrak l}$ 
is larger than the Lie algebra $C^{\infty}(S^1,{\mathfrak g})$ which corresponds
to those elements of ${\mathfrak l}$ for which all derivatives have the same
boundary value in $0$ and $1$. 

Let $\kappa:{\mathfrak g}\times{\mathfrak g}\to{\mathfrak z}$ be an invariant
bilinear form. The invariance condition means that for all $x,y,z\in
{\mathfrak g}$, we have
$$\kappa([x,y],z)=\kappa(x,[y,z]).$$
We consider in the following ${\mathfrak z}$ as a trivial 
$P({\mathfrak g})$-module. Then the Lie algebra ${\mathfrak l}$ has a central
extension $\hat{\mathfrak l}:={\mathfrak z}\times_{\omega}{\mathfrak l}$ where
the cocycle $\omega$ is given by
$$\omega(\xi,\eta):=\int_0^1\kappa(\xi,\eta'):=\int_0^1\kappa(\xi,\eta')(t)
dt.$$
We define $\widetilde{\omega}\in C^2(P({\mathfrak g}),{\mathfrak z})$ by      
\begin{eqnarray*}
\widetilde{\omega}(\xi,\eta)&:=&\frac{1}{2}\int_0^1(\kappa(\xi,\eta')-
\kappa(\eta,\xi'))=\frac{1}{2}\int_0^1(2\kappa(\xi,\eta')-\kappa(\eta,\xi)')
=\\
&=&\int_0^1\kappa(\xi,\eta')-\frac{1}{2}\kappa(\xi,\eta)(1).
\end{eqnarray*}  
We observe that for $(\xi,\eta)\in P({\mathfrak g})\times{\mathfrak l}$, we 
have $\widetilde{\omega}(\xi,\eta)=\theta(\xi,\eta):=\int_0^1
\kappa(\xi,\eta')$. This is the map $\theta$ which we used in Equation 
(\ref{definition_of_theta})
in order to associate a $3$-class to a crossed module. It is therefore 
clear that $\widetilde{\omega}$ 
may serve as $\tilde{\theta}$, the skewsymmetric extension of $\theta$ to  
$P({\mathfrak g})\times P({\mathfrak g})$.

For the following construction of the crossed module related to the string Lie
algebra, we compute:
\begin{eqnarray*}
\sum_{\rm cycl.}\int_0^1\kappa([\xi,\eta],\zeta')&=&\int_0^1\kappa([\xi,\eta],
\zeta')+\kappa([\eta,\zeta],\xi')+\kappa([\zeta,\xi],\eta')=\\
&=& \int_0^1\kappa([\xi,\eta],\zeta')+
\kappa([\xi',\eta],\zeta)+\kappa([\xi,\eta'],\zeta)=\\
&=&\int_0^1\kappa([\xi,\eta],\zeta)'=\kappa([\xi,\eta],\zeta)(1).
\end{eqnarray*}
This implies 
\begin{eqnarray*}
\sum_{\rm cycl.}\int_0^1\kappa([\xi,\eta]',\zeta)&=&\sum_{\rm cycl.}
\int_0^1\kappa([\xi',\eta],\zeta)+\kappa([\xi,\eta'],\zeta)=\\
&=&\sum_{\rm cycl.}
\int_0^1\kappa(\xi',[\eta,\zeta])+\kappa([\zeta,\xi],\eta')=\\
&=&\kappa([\eta,\zeta],\xi)(1)+\kappa([\zeta,\xi],\eta)(1)=\\
&=&2\kappa([\eta,\zeta],\xi)(1)=2\kappa([\xi,\eta],\zeta)(1).
\end{eqnarray*}
Here we used invariance of $\kappa$ and skewsymmetry of $[,]$ to rearrange 
the three terms to a complete derivative. We deduce from this computation
\begin{eqnarray*}
(d^{P({\mathfrak g})}\tilde{\omega})(\xi,\eta,\zeta)&=&\frac{1}{2}\int_0^1
\sum_{\rm cycl.}\Big(\kappa(\xi,[\eta,\zeta]')-\kappa([\eta,\zeta],\xi')\Big)\\
&=&\frac{1}{2}\Big(2\kappa([\eta,\zeta],\xi)(1)-\kappa([\eta,\zeta],\xi)(1)\Big)
=\frac{1}{2}\kappa([\eta,\zeta],\xi)(1),
\end{eqnarray*}
and the expression $d^{P({\mathfrak g})}\tilde{\omega}$ thus vanishes on 
$\Big(P({\mathfrak g})\Big)^2\times {\mathfrak l}$. This is (by restriction 
of the arguments) the cocycle identity for the cochain 
$\theta\in C^1(P({\mathfrak g}),{\mathfrak z})$ and the cochain 
$\theta\in C^0(P({\mathfrak g}),\Hom({\mathfrak l},{\mathfrak z}))=
\Hom({\mathfrak l},{\mathfrak z})$. 

By Lemma \ref{lemma_conditions_on_theta}, the formula
$$x\cdot(z,l)=(\theta(x,l),[x,l])$$
thus defines a representation of $P({\mathfrak g})$ on ${\mathfrak l}$. 
Moreover, the map $\mu:\hat{{\mathfrak l}}\to P({\mathfrak g})$ defines a 
crossed module, which we call the {\it string Lie algebra crossed module}.
By our computations, we have computed the $3$-cocycle $\chi$ defined by
${\rm ev}_1^*\chi=d^{P({\mathfrak g})}\tilde{\omega}$. We have
$$\chi(x,y,z):=\frac{1}{2}\kappa([x,y],z),$$
and $\chi\in Z^3({\mathfrak g},{\mathfrak z})$. Up to the factor $\frac{1}{2}$,
the cocycle $\chi$ is called the {\it Cartan cocycle}. The cocycle $\chi$ 
represents the class in $H^3({\mathfrak g},{\mathfrak z})$ of the crossed 
module 
$\mu:\hat{{\mathfrak l}}\to P({\mathfrak g})$. In case $\kappa\not=0$, the 
class $[\chi]\in H^3({\mathfrak g},{\mathfrak z})$ is non zero.   
For a simple Lie algebra ${\mathfrak g}$, taking $\kappa$ to be the 
{\it Killing form}, the $1$-dimensional cohomology space 
$H^3({\mathfrak g},\R)$ is generated by the class $[\chi]$.   

\begin{rem}
This example should be compared to the example of the string group 
after Baez-Crans-Schreiber-Stevenson in \cite{BCSS}.
The string Lie algebra here is the Lie algebra crossed module corresponding
to the Lie group crossed module given by the string group.
\end{rem}

In the following, our aim is to construct an {\it abelian representative} of the 
equivalence class of the crossed module associated to the string Lie algebra.
This will play a role in a later section for the construction of quasi-invariant 
tensors.    

\section{Lifting $3$-cocycles with values in $\mathbb{C}$}

\subsection{Preliminairies}

We denote by $U\!\mathfrak{g}$ the universal enveloping algebra of a Lie algebra 
$\mathfrak{g}$, and by $U\!\mathfrak{g}^+$ the kernel of the augmentation map
\[
\varepsilon:U\!\mathfrak{g}\to \mathbb{C}
\]
Identifying $\mathbb{C}$ with its dual, this gives rise to a short exact sequence of 
(left) $\mathfrak{g}$-modules:

\begin{equation}
0\to \mathbb{C}\to (U\!\mathfrak{g})^\vee \to (U\!\mathfrak{g}^+)^\vee \to 0
\label{eq:short-exact-sequence-augmentation}
\end{equation}

Recall that $U\!\mathfrak{g}$ is a cocommutative Hopf algebra. In particular, 
its vector space of endomorphisms $\text{\rm End}_\mathbb{C}(U\!\mathfrak{g})$ 
is equipped with the so called {\it convolution product} that we denote $\star$. 
The $n$-fold iterated coproduct of an element $x$ of $U\!\mathfrak{g}$ will be written
\[
\sum_{(x)} x^{(1)}\otimes \cdots \otimes x^{(n)}\quad \in\: U\!\mathfrak{g}^{\otimes n}
\]
in Sweedler's notation.

$\mathbb{C}[t]$ denotes the $\mathbb{C}$-algebra of polynomials in the variable $t$ 
with complex coefficients. If $\lambda$ is a complex number, and $V$ is a vector space, 
we have an evaluation map $\text{\rm ev}_\lambda: V\otimes \mathbb{C}[t]\to V$ which 
sends $v\otimes t^n$ to $\lambda^nv$, for all $v$ in $V$. In particular, every
$\mathbb{C}[t]$-linear map $B:V\otimes \mathbb{C}[t]\to V\otimes \mathbb{C}[t]$
is uniquely determined by its $\mathbb{C}$-linear restriction to $V$: 
$B_t:V\to V\otimes\mathbb{C}[t]$. We will employ the notation $B_{\lambda}$
to denote the composite $\text{\rm ev}_\lambda \circ B_t :V\to V$.

\subsection{Using the contracting homotopy to lift cocycles}

Let $\mathfrak{g}$ be a Lie algebra over $\mathbb{C}$ with Lie bracket $[-,-]$, and 
denote by $\eta:\mathbb{C}\to U\!\mathfrak{g}$ the unit map.
\begin{defi}
 If $x$ is an element of $U\!\mathfrak{g}$ and $g$ is an element of $\mathfrak{g}$, 
then we set
 \begin{enumerate}
 \item[(a)] $\text{\rm pr}(x):=\sum_{k\geq 0}\frac{(-1)^k}{k+1}(\text{\rm Id}-
\eta\epsilon)^{\star k}$, and this defines an endomorphism $\text{\rm pr}$ of 
$U\!\mathfrak{g}$,
 \item[(b)] $\phi_t(x):=\sum_{n\geq 0} \frac{1}{n!}\text{\rm pr}^{\star n}(x)$, which 
defines an element of $U\!\mathfrak{g}\otimes \mathbb{C}[t]$ and thus a 
$\mathbb{C}[t]$-linear endomorphism of $U\!\mathfrak{g}\otimes \mathbb{C}[t]$,
 \item[(c)] $A_t(x,g):=\sum_{(x)}\phi_{-t}(x^{(1)})\phi_t(x^{(2)}g)$.
 \end{enumerate}
\end{defi}

The following proposition summarizes the different properties of $\text{\rm pr}$, 
$\phi_t$ and $A_t$ we will need in the sequel.
\begin{prop}\label{prop-phi-t-A-t-pr}
Let $\Delta$ denote the unique $\mathbb{C}[t]$-linear coproduct of 
$U\!\mathfrak{g}\otimes \mathbb{C}[t]$ extending the comultiplication on 
$U\!\mathfrak{g}$. Then, 
\begin{enumerate}
\item $\text{\rm pr}$ is idempotent ($\text{\rm pr}\circ \text{\rm pr}=\text{\rm pr}$) 
takes its values in $\mathfrak{g}\subset U\!\mathfrak{g}$. More precisely, if 
$g_1$, ..., $g_n$ are elements of $\mathfrak{g}$, then

  \begin{equation}
\text{\rm pr}(g_1\cdots g_n)=\frac{1}{n^2}\sum_{\sigma\in\Sigma_n}(-1)^{d(\sigma)}
\binom{n-1}{d(\sigma)}^{-1}[g_{\sigma(1)},\cdots,g_{\sigma(n)}]
\label{eq:general-formula-eulerian-idem}
\end{equation}
where the notation $[h_1,\cdots,h_n]$ stands for the iterated Lie bracket 
$$[h_1,[h_2,[\cdots,[h_{n-1},h_n]\cdots ]]]$$ 
of elements $h_1$, ..., $h_n$ of 
$\mathfrak{g}$. 
\item $\phi_t$ is an endomorphism of coalgebra i.e.
  \[
  \Delta(\phi_t(x))=(\phi_t\otimes \phi_t)\Delta(x)=\sum_{(x)} \phi_t(x^{(1)})
\otimes \phi_t(x^{(2)})
\]
for all $x$ in $U\!\mathfrak{g}\subset U\!\mathfrak{g}\otimes \mathbb{C}[t]$.
\item $\phi_0=\eta\epsilon$ and $\phi_1=\text{\rm Id}$,
\item $\phi_t\star \phi_{-t}=\eta\epsilon$ on $U\!\mathfrak{g}\subset U\!
\mathfrak{g}\otimes \mathbb{C}[t]$.
\item for all $x$ in $U\!\mathfrak{g}$ and $g$ in $\mathfrak{g}$, $A_t(x,g)$ is an element of $\mathfrak{g}$ and
\[
A_t(x,g)=-\sum_{(x)} \phi_{-t}(x^{(1)}g)\phi_{t}(x^{(2)})
\]
\end{enumerate}
\end{prop}\medskip

Let $f:\Lambda^3\mathfrak{g}\to \mathbb{C}$ be a Chevalley--Eilenberg $3$-cochain 
of $\mathfrak{g}$ with values in $\mathbb{C}$. Using the identification
\[
C^2(\mathfrak{g};(U\!\mathfrak{g})^\vee)\cong C_2(\mathfrak{g};U\!\mathfrak{g})^\vee
\]
we can use $f$ to produce a $2$-cochain $\alpha$ in $C^2(\mathfrak{g};
(U\!\mathfrak{g})^\vee)$ by setting:

\begin{equation}
\tilde{\alpha}(x,g_1,g_2):=\sum_{(x)}\int_0^1 dt\: f\big(\text{\rm pr}(x^{(1)}),
A_t(x^{(2)},g_1),A_t(x^{(3)},g_2)\big)\label{eq:formule-definition-alpha}
\end{equation}
for all $x$ in $U\!\mathfrak{g}$, $g_1$ and $g_2$ in $\mathfrak{g}$. This 
$2$-cochain $\tilde{\alpha}$ induces a $2$-cochain $\alpha$ in $C^2(\mathfrak{g};
(U\!\mathfrak{g}^+)^\vee)$ by restriction to $U\!\mathfrak{g}^+\subset U\!\mathfrak{g}$.

\begin{ex}\label{example-formula-alpha-low-length}
If $x=\lambda$ is an element of $\mathbb{C}\subset U\!\mathfrak{g}$ then formula 
(\ref{eq:formule-definition-alpha}) reads
\[
\tilde{\alpha}(\lambda,g_1,g_2)=0
\]
If $x=g$ is an element of $\mathfrak{g}\subset U\!\mathfrak{g}$, then formula 
(\ref{eq:formule-definition-alpha}) reads
\[
\alpha(g,g_1,g_2)=\int_0^1dt\: f(g,tg_1,tg_2)=\frac{1}{3}f(g,g_1,g_2)
\]
If $x=gh$ is a product of two elements of $\mathfrak{g}$, then formula 
(\ref{eq:formule-definition-alpha}) reads
\begin{align*}
\alpha(gh,g_1,g_2)=&\int_0^1dt\: \big(\frac{1}{2}f([g,h],tg_1,tg_2)+
f(g,\frac{(t^2-t)}{2}[h,g_1],tg_2)\\
&+f(g,tg_1,\frac{(t^2-t)}{2}[h,g_2])+f(h,\frac{(t^2-t)}{2}[g,g_1],tg_2)\\
&+f(h,tg_1,\frac{(t^2-t)}{2}[g,g_2])\big)\\
=&\frac{1}{12}\big(2f([g,h],g_1,g_2)-f(g,[h,g_1],g_2)-f(g,g_1,[h,g_2])\\
&-f(h,[g,g_1],g_2)-f(h,g_1,[g,g_2])\big)
\end{align*}
\end{ex}

\begin{prop}
  If $f$ is a $3$-cocycle with values in $\C$, then the restriction of 
$d_{C\!E}(\tilde{\alpha})$ to $\mathbb{C}\otimes \Lambda^3\mathfrak{g}\subset 
U\!\mathfrak{g}\otimes \Lambda^3\mathfrak{g}$ is $f$. In particular, $\alpha$ 
is a $2$-cocycle in $C^2(\mathfrak{g};U\!\mathfrak{g}^+)$ and $[f]=\partial 
[\alpha]$ in $H^3(\mathfrak{g};\mathbb{C})$, where $\partial$ is the 
connecting homomorphism of the long exact sequence
\begin{equation}
\cdots\to H^2(\mathfrak{g};\mathbb{C})\to H^2(\mathfrak{g};(U\!\mathfrak{g})^\vee)
\to H^2(\mathfrak{g};(U\!\mathfrak{g}^+)^\vee)\overset{\partial}{\to} 
H^3(\mathfrak{g};\mathbb{C})\to H^3(\mathfrak{g};(U\!\mathfrak{g})^\vee)\to  \cdots
\label{eq:long-exact-sequence-augmentation}
\end{equation}
associated to the short exact sequence of coefficients 
(\ref{eq:short-exact-sequence-augmentation}).
\end{prop}
\begin{proof}
  This assertion is a consequence of a more general statement asserting that 
the right hand side of formula (\ref{eq:formule-definition-alpha}) is just 
the degree $3$ component of the dual $s_*^\vee$ of certain contracting homotopy 
$s_*$ of $(C_*(\mathfrak{g};U\!\mathfrak{g}),d_{C\!E})$ applied to $f$. However, 
for the sake of completeness, let us reprove this particular case here: For 
$x$ in $U\!\mathfrak{g}$ and $g_1$, $g_2$, $g_3$ in $\mathfrak{g}$, we have
  \begin{align*}
d_{C\!E}(\alpha)(x,g_1,g_2,g_3)=&\alpha(xg_1,g_2,g_3)-\alpha(xg_2,g_1,g_3)+\alpha(xg_3,g_1,g_2)\\
&+\alpha(x,[g_1,g_2],g_3)-\alpha(x,[g_1,g_3],g_2)+\alpha(x,g_1,[g_2,g_3])
  \end{align*}

Using the Hopf relation on $U\!\mathfrak{g}$ and the fact that the $g_i$'s
are primitive, the first part of the preceeding equality reads
\begin{align*}
\alpha(xg_1,g_2,g_3)-\alpha(xg_2,g_1,g_3)+\alpha(xg_3,g_1,g_2)=\sum_{(x)}\int_0^1 dt\:  \\
 f\big(\text{\rm pr}(x^{(1)}g_1),A_t(x^{(2)}\!\!,g_2),A_t(x^{(3)}\!\!,g_3)\big)+
f\big(\text{\rm pr}(x^{(1)}),A_t(x^{(2)}g_1,g_2),A_t(x^{(3)}\!\!,g_3)\big) +\\
+f\big(\text{\rm pr}(x^{(1)}),A_t(x^{(2)}\!\!,g_2),A_t(x^{(3)}g_1,g_3)\big)
- f\big(\text{\rm pr}(x^{(1)}g_2),A_t(x^{(2)}\!\!,g_1),A_t(x^{(3)}\!\!,g_3)\big)-\\
f\big(\text{\rm pr}(x^{(1)}),A_t(x^{(2)}g_2,g_1),A_t(x^{(3)}\!\!,g_3)\big) 
-f\big(\text{\rm pr}(x^{(1)}),A_t(x^{(2)}\!\!,g_1),A_t(x^{(3)}g_2,g_3)\big)\\
+ f\big(\text{\rm pr}(x^{(1)}g_3),A_t(x^{(2)}\!\!,g_1),A_t(x^{(3)}\!\!,g_2)\big)+
f\big(\text{\rm pr}(x^{(1)}),A_t(x^{(2)}g_3,g_1),A_t(x^{(3)}\!\!,g_2)\big) \\
+f\big(\text{\rm pr}(x^{(1)}),A_t(x^{(2)}\!\!,g_1),A_t(x^{(3)}g_3,g_2)\big)
\end{align*}

But for $y$ in $U\!\mathfrak{g}$ and $g$, $h$ in $\mathfrak{g}$, properties 
{\it 4.} and {\it 5.} of $\phi_t$ and $A_t$ listed in Proposition \ref{prop-phi-t-A-t-pr} 
imply that
\begin{align*}
A_t(yg,h)-A_t(yh,g)=&\sum_{(y)}\phi_{-t}(y^{(1)}g)\phi_t(y^{(2)}h)-\phi_{-t}(y^{(1)}h)
\phi_t(y^{(2)}g)+\\
&+\phi_{-t}(y^{(1)})\phi_t(y^{(2)}gh)-\phi_{-t}(y^{(1)})\phi_t(y^{(2)}hg)\\
=&\sum_{(y)}\phi_{-t}(y^{(1)}g)\phi_t(y^{(2)}h)-\phi_{-t}(y^{(1)}h)\phi_t(y^{(2)}g)+A_t(y,[g,h])\\
=&\sum_{(y)}\phi_{-t}(y^{(1)}g)\phi_t(y^{(2)})\phi_{-t}(y^{(3)})\phi_t(y^{(4)}h)-\\
&\phi_{-t}(y^{(1)}h)\phi_t(y^{(2)})\phi_{-t}(y^{(3)})\phi_t(y^{(4)}g)+A_t(y,[g,h])\\
=&A_t(y,[g,h])-\sum_{(y)}[A_{t}(y^{(1)},g),A_t(y^{(2)},h)]\\
\end{align*}
Thus, all terms of the form $f(\cdots, A_t(x^{(i)}g_j,g_k),\cdots)$ appearing in 
$\alpha(xg_1,g_2,g_3)-\alpha(xg_2,g_1,g_3)+\alpha(xg_3,g_1,g_2)$ can be sorted 
by pairs $(j,k)$ to cancel the corresponding term of the form 
$\alpha(\cdots,[g_j,g_k],\cdots)$ in $d_{C\!E}(\alpha)(x,g_1,g_2,g_3)$. 
Hence, we are left with
\begin{align*}
  d_{C\!E}(\alpha)(x,g_1,g_2,g_3)=\sum_{(x)}\int_0^1 dt\:
f\big(\text{\rm pr}(x^{(1)}g_1),A_t(x^{(2)}\!\!,g_2),A_t(x^{(3)}\!\!,g_3)\big)-\\
f\big(\text{\rm pr}(x^{(1)}),[A_t(x^{(2)}\!\!,g_1),A_t(x^{(3)}\!\!,g_2)],
A_t(x^{(4)}\!\!,g_3)\big) 
- f\big(\text{\rm pr}(x^{(1)}g_2),A_t(x^{(2)}\!\!,g_1),
A_t(x^{(3)}\!\!,g_3)\big)\\
+f\big(\text{\rm pr}(x^{(1)}),A_t(x^{(2)}\!\!,g_1),[A_t(x^{(3)},g_2)A_t(x^{(4)},g_3)]\big)
+ f\big(\text{\rm pr}(x^{(1)}g_3),A_t(x^{(2)}\!\!,g_1),A_t(x^{(3)}\!\!,g_2)\big)\\
+f\big(\text{\rm pr}(x^{(1)}),[A_t(x^{(2)}\!\!,g_1),A_t(x^{(3)}\!\!,g_3)],
A_t(x^{(4)}\!\!,g_2)\big)
\end{align*}
Since $f$ is a cocycle, we have that
\[
d_{C\!E}(f)\big(\text{\rm pr}(x^{(1)}),A_t(x^{(2)},g_1),A_t(x^{(3)},g_2),A_t(x^{(4)},g_3)\big)=0
\]
which enables to rewrite the three terms of $d_{C\!E}(\alpha)(x,g_1,g_2,g_3)$ 
involving Lie brackets to obtain
\begin{align*}
  d_{C\!E}(\alpha)(x,g_1,g_2,g_3)=\sum_{(x)}\int_0^1 dt\:
f\big(\text{\rm pr}(x^{(1)}g_1),A_t(x^{(2)}\!\!,g_2),A_t(x^{(3)}\!\!,g_3)\big)-\\
f\big([\text{\rm pr}(x^{(1)}),A_t(x^{(2)}\!\!,g_1)],A_t(x^{(3)}\!\!,g_2),A_t(x^{(4)}\!\!,g_3)\big) 
-f\big(\text{\rm pr}(x^{(1)}g_2),A_t(x^{(2)}\!\!,g_1),A_t(x^{(3)}\!\!,g_3)\big)\\
+f\big([\text{\rm pr}(x^{(1)}),A_t(x^{(2)}\!\!,g_2)],A_t(x^{(3)},g_1),A_t(x^{(4)},g_3)\big)
+ f\big(\text{\rm pr}(x^{(1)}g_3),A_t(x^{(2)}\!\!,g_1),A_t(x^{(3)}\!\!,g_2)\big)-\\
f\big([\text{\rm pr}(x^{(1)}),A_t(x^{(2)}\!\!,g_3)],A_t(x^{(3)}\!\!,g_1),A_t(x^{(4)}\!\!,g_2)\big)
\end{align*}
But one easily checks that for all $y$ in $U\!\mathfrak{g}$ and $g$ in $\mathfrak{g}$
\[
\text{\rm pr}(yg)-\sum_{(y)} [\text{\rm pr}(y^{(1)}),A_t(y^{(2)},g)]=\frac{d}{dt}A_t(y,g)
\]
so that, using point {\it 3.} of Proposition \ref{prop-phi-t-A-t-pr}
\begin{align*}
  d_{C\!E}(\alpha)(x,g_1,g_2,g_3)=\sum_{(x)}\int_0^1 dt\: 
f\big(\frac{d}{dt}A_t(x^{(1)},g_1),A_t(x^{(2)}\!\!,g_2),A_t(x^{(3)}\!\!,g_3)\big)+\\
f\big(A_t(x^{(1)},g_1),\frac{d}{dt}A_t(x^{(2)}\!\!,g_2),A_t(x^{(3)}\!\!,g_3)\big)
+f\big(A_t(x^{(1)},g_1),A_t(x^{(2)}\!\!,g_2),\frac{d}{dt}A_t(x^{(3)}\!\!,g_3)\big)\\
=\sum_{(x)}\int_0^1 dt\: \frac{d}{dt}
f\big(A_t(x^{(1)}\!\!,g_1),A_t(x^{(2)}\!\!,g_2),A_t(x^{(3)}\!\!,g_3)\big)\\
=\sum_{(x)}f\big(A_1(x^{(1)}\!\!,g_1),A_1(x^{(2)}\!\!,g_2),A_1(x^{(3)}\!\!,g_3)\big)-
f\big(A_0(x^{(1)}\!\!,g_1),A_0(x^{(2)}\!\!,g_2),A_0(x^{(3)}\!\!,g_3)\big)\\
= \epsilon(x) f(g_1,g_2,g_3)
\end{align*}
from which the proposition follows immediately.
\end{proof}

 Together with Theorem \ref{construction_thm}, this shows:

\begin{cor}
  If $f$ is a $3$-cocycle, then the Yoneda product
\[
0\to \mathbb{C}\to (U\!\mathfrak{g})^\vee \to (U\!\mathfrak{g}^+)^\vee\oplus_\alpha \mathfrak{g} 
\to \mathfrak{g}\to 0
\]
is a crossed module of Lie algebras whose equivalence class represents $[f]$.
\end{cor}

In other words, we have constructed abelian representatives (of the equivalence 
class of crossed modules associated) to each 3-cohomology class given by 
some constant-valued 3-cocycle $f$. 

\subsection{The example of the Cartan cocycle of $\mathfrak{sl}_2(\mathbb{C})$}

In this subsection, $\mathfrak{g}:=\mathfrak{sl}_2(\mathbb{C})$ and $f:\Lambda^3\mathfrak{g}\to\mathbb{C}$ is the
Cartan $3$-cocycle defined by
\[
f(g_1,g_2,g_3):=\kappa(g_1,[g_2,g_3])
\]
for all $g_1$, $g_2$ and $g_3$ in $\mathfrak{g}$, where $\kappa$ is the Killing form
of $\mathfrak{g}=\mathfrak{sl}_2(\mathbb{C})$ (In fact, $f$ is, up to a non-zero scalar, the only non-trivial possible $3$-cochain). Our goal is to study the $2$-cocycle $\alpha$ lifting $f$ in order to get a concrete description of the crossed module 
encoding the Cartan cocycle. More precisely, we want to determine the values of
\[
\alpha(x,g_1,g_2):=\int_0^1dt\: \kappa(\text{\rm pr}(x^{(1)}),[A_t(x^{(2)},g_1),A_t(x^{(3)},g_2)])
\]
when $g_1$, $g_2$ run over a given basis of $\mathfrak{g}$ and $x$ runs over the 
associated PBW basis of $U\!\mathfrak{g}^+$.
Denote by $(X,Y,H)$ the standard ordered basis of $\mathfrak{sl}_2(\mathbb{C})$ 
which satisfies the relations
\[
[X,Y]=H\:,\qquad [H,X]=2X\:,\qquad [H,Y]=-2Y.
\]
In this basis, the problem amounts to determine the values of the following complex numbers
\[
B_{XY}(a,b,c):=\alpha(X^aY^bH^c,X,Y)
\]
and
\[
B_{XH}(a,b,c):=\alpha(X^aY^bH^c,X,H)
\]
and
\[
B_{YH}(a,b,c):=\alpha(X^aY^bH^c,Y,H)
\]
for all natural numbers $a$, $b$ and $c$.

If obtaining a general formula for the values of $B_{XY}$, $B_{XH}$ and $B_{YH}$ seems out of reach for the moment, in particular because the combinatorics appearing in the expression (\ref{eq:general-formula-eulerian-idem}) of $\text{\rm pr}(x)$, for a word $x$ of $U\!\mathfrak{g}$, get more and more complicated as the length of $x$ increases, we can at least show that these values are ``often'' $0$.

\begin{prop}\label{proposition-vanishing-cartan-sl2}Let $a$, $b$ and $c$ be three natural numbers.
  \begin{enumerate}
  \item If $a\neq b$, then
\[
B_{XY}(a,b,c)=0
\]
  \item If $a\neq b-1$, then
\[
B_{XH}(a,b,c)=0
\]
  \item If $a\neq b+1$, then
\[
B_{YH}(a,b,c)=0
\]
  \end{enumerate}
\end{prop}

Before entering the proof of the proposition, let us establish the following lemma which shows that non-trivial iterated brackets in $\mathfrak{sl}_2(\mathbb{C})$ must have a very constrained form.
\begin{lem}
  Let $(X_1, ...,X_n)$ be a sequence of basis vectors in $\{X,Y,H\}^n$ and denote by $a:=\text{\rm card}\{ i\:/X_i=X\}$ and $b:=\text{\rm card}\{ i\:/X_i=Y\}$ respectively the number of $X$'s and the number of $Y$'s appearing in this sequence. If the iterated bracket
\[
[X_1,\cdots,X_n]:=[X_1,[X_2,[\cdots,[X_{n-1},X_n]\cdots]]]
\]
is not zero, then
\[
|a-b|\leq 1
\]
In this case, there exists a non-zero complex number $\lambda$ such that
\[
[X_1,\cdots,X_n]=\left\{
\begin{array}{cl}
  \lambda H&\text{if $a=b$,}\\
 \lambda X &\text{if $a=b+1$,}\\
 \lambda Y &\text{if $a=b-1$.}
\end{array}
\right.
\]
\end{lem}
{\bf{Proof of the lemma.}\quad} 
Let $(X_1, ...,X_n)$ be a sequence of basis vectors such that $[X_1,\cdots,X_n]\neq 0$.
We can assume that none of the $X_i$'s is equal to $H$ since the value of an iterated bracket containing an $H$ is equal to twice the value of the iterated bracket obtained by erasing this $H$ i.e.
\[
[Z_1,\cdots,Z_i,H,Z_{i+1},\cdots, Z_m]=2[Z_1,\cdots,Z_m]
\]
for all $Z_1$, ..., $Z_m$ in $\{X,Y,H\}$. We now proceed by induction on the length $n$ of the sequence $(X_1,\cdots,X_n)$. Assume that the proposition holds for all values of $k$ strictly lower than a given $n$, in particular for $n=n-2$. Then the non-zero iterated bracket we are interested in takes the form
\[
[X_1,[X_2,[X_3,\cdots,X_n]]]\neq 0
\]
If we denote by $a'$ (resp. $b'$) the number of $X$'s (resp. the number of $Y$'s) in the sequence $(X_3,\cdots,X_N)$, applying the conclusion of the proposition (for $k=n-2$) to the necessarily non-zero bracket $[X_3,\cdots,X_n]$ allows to distinguish $3$ cases
\begin{itemize}
  \item $[X_3,\cdots,X_n]=\lambda' H$ for some non zero compex number $\lambda'$. This implies that $a'=b'$. In addition, we see that in order to have 
\[
0\neq [X_1,[X_2,[X_3,\cdots,X_n]]]=\lambda' [X_1,[X_2,H]]=2\lambda' \pm [X_1,X_2]
\]
$X_1$ and $X_2$ must be different. Hence $a=a'+1=b'+1=b$ i.e. $a=b$ and $[X_1,[X_2,[X_3,\cdots,X_n]]]=\lambda H$ with $\lambda=\pm2\lambda'$.
  \item $[X_3,\cdots,X_n]=\lambda' X$ for some non zero compex number $\lambda'$. This corresponds to the case $a'=b'+1$. In this case, we see that the only possible choice of $X_1$ and $X_2$ leading to a non zero value of $[X_1,[X_2,[X_3,\cdots,X_n]]]$ is given by taking $(X_1,X_2)=(X,Y)$, which leads to 
\[
[X_1, ...,X_n]=2\lambda' X\quad\text{and}\quad a=a'+1=b'+2=b+1
\]
or taking $(X_1,X_2)=(Y,Y)$ for which
\[
[X_1, ...,X_n]=-2\lambda' Y\quad\text{and}\quad a=a'=b'+1=b-1
\]
In both cases, the conclusion of the lemma holds.
\item $[X_3,\cdots,X_n]=\lambda' Y$ for some non zero compex number $\lambda'$. In this case, $a'=b'-1$ and the only possible choices for $(X_1,X_2)$ are $(X_1,X_2)=(Y,X)$, which leads to
\[
[X_1, ...,X_n]=2\lambda' Y\quad\text{and}\quad a=b-1
\]
and $(X_1,X_2)=(X,X)$, for which
\[
[X_1, ...,X_n]=-2\lambda' X\quad\text{and}\quad a=b+1
\]
Again, we see that theses values of $a$, $b$ and $[X_1,\cdots,X_n]$ satisfy the conclusion of the lemma.
\end{itemize}
The initialization of the induction is left to the reader.
\fin

{\bf Proof of proposition \ref{proposition-vanishing-cartan-sl2}.}
In view of formula (\ref{eq:formule-definition-alpha}) defining $\alpha$ and of the definition of the coproduct on $U\!\mathfrak{g}$, it is clear that $B_{XY}(a,b,c):=\alpha(X^aY^bH^c,X,Y)$ can be written as a sum of the form
\[
B_{XY}(a,b,c)=\sum_{\bar{a},\bar{b},\bar{c}} \lambda_{\bar{a},\bar{b},\bar{c}} f(\text{\rm pr}(X^{a_1}Y^{b_1}H^{c_1}),A_t(X^{a_2}Y^{b_2}H^{c_2},X),A_t(X^{a_3}Y^{b_3}H^{c_3},Y))
\]
where the summation runs over all triples of integers $\bar{a}:=(a_1,a_2,a_3)$, $\bar{b}:=(b_1,b_2,b_3)$ and $\bar{c}:=(c_1,c_2,c_3)$ such that $a_1+a_2+a_3=a$,  $b_1+b_2+b_3=b$ and  $c_1+c_2+c_3=c$ and where the $\lambda_{\bar{a},\bar{b},\bar{c}}$'s are complex numbers. Our aim is to determine which terms of this sum have a chance to be non-zero. For each choice of triples $\bar{a}$, $\bar{b}$ and $\bar{c}$, formula (\ref{eq:general-formula-eulerian-idem}) and the lemma imply that $\text{\rm pr}(X^{a_1}Y^{b_1}H^{c_1})$ can be non-zero only when $|a_1-b_1|\leq 1$ and in this case

\begin{equation}
\text{\rm pr}(X^{a_1}Y^{b_1}H^{c_1})=\left\{
  \begin{array}{cl}
 \lambda H&\text{if $a_1=b_1$,}\\
 \lambda X &\text{if $a_1=b_1+1$,}\\
 \lambda Y &\text{if $a_1=b_1-1$.}
   \end{array}\right.\label{eq:PR-formule-cartan-SL2}
\end{equation}
for some complex number $\lambda$.
Moreover, one can check that by definition of the operator $A_t$, $A_t(X^{a_2}Y^{b_2}H^{c_2},X)$ is a linear combination of words of length $a_1+a_2+c_2+1$ containing exactly $a_2+1$ copies of $X$, $b_2$ copies of $Y$ and $c_2$ copies of $H$. Since it is in $\mathfrak{g}$ by proposition \ref{prop-phi-t-A-t-pr}, it is equal to its own image under $\text{\rm pr}$ i.e.
\[
A_t(X^{a_2}Y^{b_2}H^{c_2},X)=\text{\rm pr}(A_t(X^{a_2}Y^{b_2}H^{c_2},X))
\]
Thus, the lemma implies that $A_t(X^{a_2}Y^{b_2}H^{c_2},X)$ can be non zero only when $|a_2+1-b_2|\leq 1$ and in this case
\begin{equation}
 A_t(X^{a_2}Y^{b_2}H^{c_2},X)=\left\{
\begin{array}{cl}
 \beta H&\text{if $a_2=b_2-1$,}\\
 \beta X &\text{if $a_2=b_2$,}\\
 \beta Y &\text{if $a_2=b_2-2$.}
\end{array}\right.\label{eq:ATX-formule-cartan-SL2}
\end{equation}
for some complex number $\beta$. Similarly, one can prove that the only non-zero possible values of $A_t(X^{a_3}Y^{b_3}H^{c_3},Y)$ are given by
\begin{equation}
 A_t(X^{a_3}Y^{b_3}H^{c_3},Y)=\left\{
\begin{array}{cl}
 \mu H&\text{if $a_3=b_3+1$,}\\
 \mu X &\text{if $a_3=b_3+2$,}\\
 \mu Y &\text{if $a_3=b_3$.}
\end{array}\right.\label{eq:ATY-formule-cartan-SL2}
\end{equation}
for some complex number $\mu$. Since $f$ is skew-symmetric, the term
\[
f(\text{\rm pr}(X^{a_1}Y^{b_1}H^{c_1}),A_t(X^{a_2}Y^{b_2}H^{c_2},X),A_t(X^{a_3}Y^{b_3}H^{c_3},Y))
\]
can be non-zero only when its three arguments are all different. Equations (\ref{eq:PR-formule-cartan-SL2}), (\ref{eq:ATX-formule-cartan-SL2}) and (\ref{eq:ATY-formule-cartan-SL2}) show that this can happen only if
\[
  \begin{array}{cl}
& (a_1,a_2,a_3)=(b_1,b_2,b_3)\\
\text{or}&(a_1,a_2,a_3)=(b_1,b_2-2,b-3+2)\\
\text{or}&(a_1,a_2,a_3)=(b_1+1,b_2-1,b_3)\\
\text{or}&(a_1,a_2,a_3)=(b_1+1,b_2-2,b_3+1)\\
\text{or}&(a_1,a_2,a_3)=(b_1-1,b_2-1,b_3+2)\\
\text{or}&(a_1,a_2,a_3)=(b_1-1,b_2,b_3+1)\\
  \end{array}
\]
Each of these six cases implies that $a_1+a_2+a_3=b_1+b_2+b_3$ i.e. $a=b$, as wanted to prove point $1.$ of the proposition. The two other points can be established exactly in the same fashion.
\fin

\begin{ex}\label{example2-calcul-alphaXY-X-Y-SL2}
 Using the formula for $\alpha(gh,g_1,g_2)$ given at the end of example \ref{example-formula-alpha-low-length}, we see that if $(g_1,g_2)=(X,Y)$, then the only non trivial value $\alpha(gh,g_1,g_2)$ is given by
\[
\alpha(XY,X,Y)=\frac{1}{12}(2f(H,X,Y)+f(X,H,Y)-f(Y,X,H))=\frac{1}{6}f(X,Y,H)=\frac{4}{3}
\]
\end{ex}

\section{Application}

Here we apply the construction of the abelian representative of the equivalence class
of the crossed module associated with a 3-cohomology class with values in the trivial
module to the construction of quasi-invariant tensors, see \cite{CirMar}. 

\subsection{Preliminairies}

Let $\mu:{\mathfrak m}\to{\mathfrak n}$ be a crossed module. The diagonal map 
$U{\mathfrak n}\to U({\mathfrak n}^{\oplus n})$, $x\mapsto x\otimes 1\otimes\ldots\otimes 1+\ldots
+1\otimes\ldots\otimes1\otimes x$ will be denoted by $\triangle^n$. Denote by ${\mathfrak e}$ 
the semi-direct product Lie algebra ${\mathfrak m}\rtimes{\mathfrak n}$. 

The Lie algebra morphism $\beta:{\mathfrak e}\to{\mathfrak n}$ given for $X\in{\mathfrak n}$ and 
$v\in{\mathfrak m}$ by
$$\beta(v,X)\,=\,\mu(v)+X,$$
entends to an algebra morphism $\beta:U{\mathfrak e}\to U{\mathfrak n}$ and further to 
$\beta:U({\mathfrak e}^{\oplus n})\to U({\mathfrak n}^{\oplus n})$. 

Let $A_n$ be the smallest vector subspace of $U({\mathfrak e}^{\oplus n})$ containing all 
elements of the form $xvy$ with $x,y\in U({\mathfrak n}^{\oplus n})$ and 
$v\in{\mathfrak m}^{\oplus n}$. 

\begin{defi}
The vector space ${\mathcal U}^{(n)}$ is the quotient of $A_n$ with respect to the relations
$$x\mu(u)yvz\,=\,xuy\mu(v)z$$
for all $x,y,z\in U({\mathfrak n}^{\oplus n})$ and all $u,v\in{\mathfrak m}^{\oplus n}$.
\end{defi}
We suggest that the reader reads some examples in Section 4.1 of \cite{CirMar} in order
to better understand this notion. The upshot of the vector space ${\mathcal U}^{(n)}$ is 
the following Lemma:

\begin{lem}[Lemma 27 in \cite{CirMar}]   \label{(66)}
For all $a,b\in{\mathcal U}^{(n)}$, we have 
$$\beta(a)b\,=\,a\beta(b).$$
\end{lem}       

Finally, let us introduce an action of ${\mathfrak m}$ on ${\mathfrak n}\otimes{\mathfrak n}$:
For all $a\in{\mathfrak m}$ and all $r=\sum_is_i\otimes t_i\in{\mathfrak n}\otimes{\mathfrak n}$,
we set
$$a\cdot r\,=\,-\sum_i (s_i\cdot a)\otimes t_i +s_i\otimes(t_i\cdot a),$$
where $s_i\cdot a$ is the action of ${\mathfrak n}$ on ${\mathfrak m}$ stemming from the 
crossed module $\mu:{\mathfrak m}\to{\mathfrak n}$. This definition is consistent with
seeing all elements as elements of the semi-direct product and acting by the bracket.   

Now we are ready for the main definition:

\begin{defi}  \label{quasiinvariant_tensor}
A quasi-invariant tensor in $\mu:{\mathfrak m}\to{\mathfrak n}$ is a triple $(r,c,\xi)$
where:
\begin{enumerate}
\item[(a)] $r=\sum_qs_q\otimes t_q$ is a symmetric tensor (i.e. 
$\sum_qs_q\otimes t_q=\sum_qt_q\otimes s_q$);
\item[(b)] $\xi:{\mathfrak n}\to({\mathfrak n}\otimes{\mathfrak m})\oplus
({\mathfrak m}\otimes{\mathfrak n})$ is a linear map, whose image is symmetric, i.e.
$$\xi(X)\,=\,\sum_a\xi_a(X)\mu_a'\otimes\mu_a''+\sum_a\xi_a(X)\mu_a''\otimes\mu_a'$$
with $\mu_a'\in{\mathfrak n}$ and $\mu_a''\in{\mathfrak m}$. 
\item[(c)] $c$ is an ${\mathfrak n}$-invariant element in $\ker(\mu)\subset{\mathfrak m}$. 
\end{enumerate}
These data is supposed to satisfy the following conditions:
\begin{enumerate}
\item[(a)] $X\cdot r\,=\,\beta(\xi(X))$ for all $X\in{\mathfrak n}$,
\item[(b)] $u\cdot r\,=\,\xi(\mu(u))$ for all $u\in {\mathfrak m}$,
\item[(c)] $\xi([X,Y])\,=\,X\cdot\xi(Y)-Y\cdot\xi(X)$ for all $X,Y\in{\mathfrak n}$.
\end{enumerate}
\end{defi}

Cirio and Martins show in \cite{CirMar}, Theorem 31, how to construct from 
a quasi-invariant tensor a 
totally symmetric infinitesimal 2-braiding in some strict linear 2-category ${\mathcal C}_{\mu}$ 
associated with the crossed module $\mu:{\mathfrak m}\to{\mathfrak n}$.      
This is seen as a categorification of the infinitesimal braiding in the category of 
${\mathfrak n}$-modules which is associated to an r-matrix for ${\mathfrak n}$, i.e. a 
symmetric ${\mathfrak n}$-invariant tensor $r\in{\mathfrak n}\otimes{\mathfrak n}$.

Cirio and Martins go on constructing a quasi-invariant tensor for a certain crossed module
corresponding to the Cartan cocycle of ${\mathfrak s}{\mathfrak l}_2(\C)$. 
Let us show here how one may construct in general a quasi-invariant tensor using the 
abelian representatives of the equivalence class of the crossed module 
$\mu:{\mathfrak m}\to{\mathfrak n}$ which we have constructed earlier. 

\subsection{Construction of quasi-invariant tensors} \label{construction}

Suppose given a crossed module of the form $\mu:V_2\to V_3\times_{\alpha}{\mathfrak g}$, i.e. 
it is spliced together from a short exact sequence of ${\mathfrak g}$-modules 
\begin{equation}   \label{***}
0\to V=V_1\stackrel{i}{\to} V_2 \stackrel{d}{\to} V_3\to 0,
\end{equation}
and an abelian extension $V_3\times_{\alpha}{\mathfrak g}$ of ${\mathfrak g}$ by $V_3$ 
via the cocycle 
$\alpha\in Z^2({\mathfrak g},V_3)$, see Theorem \ref{construction_thm}. 
We will always suppose that $V=V_1=\C$ is the trivial 1-dimensional ${\mathfrak g}$-module.
Observe that this means that there is an element $1_{V_2}\in V_2$ which generates the trivial 
${\mathfrak g}$-submodule $i(\C)$ and the kernel of $d$. 
We denote by $Q:V_3\to V_2$ a linear section of the quotient map $d$. 
We write 
$\overline{X}=(0,X)\in V_3\times_{\alpha}{\mathfrak g}$ for all $X\in{\mathfrak g}$ and   
$\overline{h}=(h,0)\in V_3\times_{\alpha}{\mathfrak g}$ for all $h\in V_3$.
With these notations, we have for example $\mu(v)=\overline{d(v)}$ for all $v\in V_2$. 

Introduce further the cochains showing up when computing the image of $\alpha$ under the 
connecting homomorphism associated to the short exact sequence (\ref{***}). There is the 
2-cochain $\omega:\Lambda^2{\mathfrak g}\to V_2$ given for all $X,Y\in{\mathfrak g}$ 
by $Q(\alpha(X,Y))=w(X,Y)$, and further $\Phi:\Lambda^3{\mathfrak g}\to V_2$ given 
by $\Phi=d_{CE}\omega$. By definition of the connecting homomorphism (and Theorem 
\ref{construction_thm}), the 3-class $[\gamma]=\partial[\alpha]\in H^3({\mathfrak g};V)$ 
is represented by $\gamma$ such that for all $X,Y,Z\in{\mathfrak g}$, 
$i(\gamma(X,Y,Z))=\Phi(X,Y,Z)$. All this can be visualized in the following diagram:

\vspace{.3cm}
\hspace{2cm}\xymatrix{
 & \omega=Q\alpha \ar@{|-}[d] \ar@{|->}[r]^{d} & \alpha \ar@{|-}[d]
 \\
C^2(\mathfrak{g},V_1) \ar@{ >->}[r]^{i} \ar[d]^{d_{CE}} & C^2(\mathfrak{g},
V_2) \ar@{->>}[r]^{d} \ar@{-}[d]^{d_{CE}} & C^2(\mathfrak{g},V_3) 
\ar@{-}[d]^{d_{CE}} \\
C^3(\mathfrak{g},V_1) \ar@{ >->}[r]^{i}  & C^3(\mathfrak{g},V_2) 
\ar[d]\ar@{->>}[r]^{d}  & C^3(\mathfrak{g},V_3) \ar[d]\\
\partial\alpha=\gamma \ar@{|->}[r]^{i} & d_{CE}\omega=\Phi 
\ar@{|->}[r]^{d} & 0=d_{CE}\alpha}\vspace{.5cm}

The computation in Equation (80) in \cite{CirMar} goes through in this more general framework
to show that in ${\mathcal U}^{(2)}$:
$$1_{V_2}\otimes [\overline{X},\overline{Y}]\,=\,1_{V_2}\otimes\overline{[X,Y]}.$$

Consider now an r-matrix for ${\mathfrak g}$, i.e. a ${\mathfrak g}$-invariant, symmetric 
tensor $r\in{\mathfrak g}\otimes{\mathfrak g}$. We write its components as
$$r\,=\,\sum_is_i\otimes t_i.$$ 

Let us define the following {\it compatibility condition}:
We have for all $X,Y\in{\mathfrak g}$
\begin{equation}                \label{compatibility}
\sum_i\Phi(s_i,X,Y)\otimes t_i\,=\,1_{V_2}\otimes[X,Y],\,\,\,\,{\rm and}\,\,\,\,     
\sum_is_i\otimes\Phi(t_i,X,Y)\,=\,[X,Y]\otimes 1_{V_2}.
\end{equation}

Now in order to construct a quasi-invariant tensor, we consider the following objects:

\begin{enumerate}
\item[(a)] $\overline{r}=\sum_i\overline{s}_i\otimes\overline{t}_i$ the lift of $r$ to 
$V_3\times_{\alpha}{\mathfrak g}=:E$;
\item[(b)] the map $\xi:E\to (E\otimes V_2)\oplus(V_2\otimes E)$ defined as $\xi=-\xi_0-C$
where 
\begin{enumerate}
\item[($\alpha)$] $\xi_0(\overline{X})\,=\,\sum_i\omega(s_i,X)\otimes \overline{t}_i+
\overline{s}_i\otimes \omega(t_i,X)$,
\item[($\beta$)] $\xi_0(\overline{h})\,=\,\sum_i(s_i\cdot Q(h))\otimes\overline{t}_i
+\overline{s}_i\otimes(t_i\cdot Q(h))$, and 
\item[($\gamma$)] $C(h,X)\,=\,1_{V_2}\otimes\overline{X}+\overline{X}\otimes 1_{V_2}$.
\end{enumerate}
\item[(c)] $c\in\ker(\mu)=\C\subset V_2$ (via the embedding $i$).
\end{enumerate}

\begin{theo}\label{theorem-compat-implies-quasi-invariance}
Suppose that in the above situation the compatibility condition
(\ref{compatibility}) is satisfied. 

Then the triple $(\overline{r},\xi,c)$ is a quasi-invariant tensor for the crossed module 
$\mu:V_2\to V_3\times_{\alpha}{\mathfrak g}$.
\end{theo}

The proof of the theorem is a simply adaptation of the proof 
of Theorem 35 in \cite{CirMar}: 

\pr We need to check conditions {\it (a)-(c)} in Definition \ref{quasiinvariant_tensor}. 
We start with condition {\it (a)}:
\begin{eqnarray*}
\overline{X}\cdot\overline{r} &=& \sum_i\big((\alpha(X,s_i),[X,s_i])\big)\otimes \overline{t}_i+
\sum_i\overline{s}_i\otimes\big((\alpha(X,t_i),[X,t_i])\big) \\
&=& \sum_i\overline{\alpha(X,s_i)}\otimes\overline{t}_i+ \sum_i\overline{s}_i\otimes
\overline{\alpha(X,t_i)} \\
&=& -\beta(\xi_0(\overline{X}))\,=\,\beta(\xi(\overline{X})),
\end{eqnarray*}
where we have used in this order the definition of the bracket in the abelian 
extension, the invariance of $r$, the relation of $\alpha$ to $\omega$ (more precisely, 
$$\beta(\omega(s_i,X)\otimes \overline{t}_i)=\mu(\omega(s_i,X))\otimes\overline{t}_i=
-\overline{\alpha(X,s_i)}\otimes\overline{t}_i\hspace{.5cm}),$$ 
and finally that $\mu(1_{V_2})=0$. 

\begin{eqnarray*}
\overline{h}\cdot\overline{r} &=& - \sum_i\big( \overline{s_i\cdot h}\otimes \overline{t}_i+
\overline{s}_i\otimes\overline{t_i\cdot h} \big) \\
&=& - \sum_i\beta\big(  (s_i\cdot Q(h))\otimes \overline{t}_i+
\overline{s}_i\otimes (t_i\cdot Q(h))\big)\,=\,-\beta(\xi_0(\overline{h}))\,=\,
\beta(\xi(\overline{h})), 
\end{eqnarray*}
where we have used in this order the fact that the bracket in the semi-direct product is here 
just given by the action, then using axiom {\it (a)} of a crossed module
$$\beta((s_i\cdot Q(h)\otimes \overline{t}_i)\,=\,\mu(s_i\cdot Q(h))\otimes \overline{t}_i
\,=\,[s_i,h]\otimes \overline{t}_i\,=\,(s_i\cdot h)\otimes \overline{t}_i,$$
 and finally the definitions of $\xi_0$ and $\xi$.

Now let us verify condition {\it (b)}: For all $a\in V_2$, we have 
$$a\cdot\overline{r}-\xi(\mu(a))\,=\,\sum_i (s_i\cdot(Q(d(a))-a))\otimes \overline{t}_i+
\overline{s}_i\otimes (t_i\cdot(Q(d(a))-a)),$$
and this vanishes since $Q(d(a))-a$ is in $\ker(d)=\C$ which is ${\mathfrak g}$-invariant. 

Finally, we verify condition {\it (c)}: For this, we start by computing:

\begin{eqnarray*}
\xi_0([\overline{X},\overline{Y}])&=&\xi_0(\overline{[X,Y]})-\xi_0(\overline{\alpha(X,Y)}) \\
&=&\sum_i\big(\omega(s_i,[X,Y])\otimes\overline{t}_i+\overline{s}_i\otimes \omega(t_i,[X,Y]) + \\
&+&(s_i\cdot Q\alpha(X,Y))\otimes \overline{t}_i+\overline{s}_i\otimes(t_i\cdot Q\alpha(X,Y))\big)
\end{eqnarray*}

$$\overline{X}\cdot\xi_0(\overline{Y})\,=\,\sum_i(X\cdot \omega(s_i,Y))\otimes \overline{t}_i
+\omega(s_i,Y)\otimes[\overline{X},\overline{t}_i]+[\overline{X},\overline{s}_i]\otimes \omega(t_i,Y)
+\overline{s}_i\otimes(X\cdot \omega(t_i,Y))$$

$$\overline{Y}\cdot\xi_0(\overline{X})\,=\,\sum_i(Y\cdot \omega(s_i,X))\otimes \overline{t}_i
+\omega(s_i,X)\otimes[\overline{Y},\overline{t}_i]+[\overline{Y},\overline{s}_i]\otimes \omega(t_i,X)
+\overline{s}_i\otimes(Y\cdot \omega(t_i,X))$$

We now claim that 
\begin{eqnarray*}
\xi_0([\overline{X},\overline{Y}])-\overline{X}\cdot\xi_0(\overline{Y})+
\overline{Y}\cdot\xi_0(\overline{X})&=&\sum_i\Phi(s_i,X,Y)\otimes\overline{t}_i+
\overline{s}_i\otimes\Phi(t_i,X,Y) \\
&=& \overline{[X,Y]}\otimes 1_{V_2}+1_{V_2}\otimes\overline{[X,Y]}.
\end{eqnarray*}
where we used the compatibility condition (\ref{compatibility}) to rewrite the last expression.

Reminding that $\Phi=d_{CE}\omega$, we obtain
\begin{eqnarray*}
\xi_0([\overline{X},\overline{Y}])-\overline{X}\cdot\xi_0(\overline{Y})+
\overline{Y}\cdot\xi_0(\overline{X}) -\sum_i\Phi(s_i,X,Y)\otimes\overline{t}_i+
\overline{s}_i\otimes\Phi(t_i,X,Y) = \\
= \sum_i -\omega(s_i,Y)\otimes[\overline{X},\overline{t}_i]-[\overline{X},\overline{s}_i]\otimes
\omega(t_i,Y)+\omega(s_i,X)\otimes[\overline{Y},\overline{t}_i]+[\overline{Y},\overline{s}_i]\otimes 
\omega(t_i,X)+\\
-\omega(X,[Y,s_i])\otimes\overline{t}_i-\omega(Y,[s_i,X])\otimes\overline{t}_i-
\overline{s}_i\otimes \omega(X,[Y,t_i])-\overline{s}_i\otimes \omega(Y,[t_i,X]).
\end{eqnarray*}

We now expand each term of the form $[\overline{X},\overline{t}_i]$ into
$\overline{\alpha(X,t_i)}+\overline{[X,t_i]}$. The terms involving the cocycle $\alpha$ are:
$$\sum_i -\omega(s_i,Y)\otimes\overline{\alpha(X,t_i)}-\overline{\alpha(X,s_i)}\otimes \omega(t_i,Y)
+\omega(s_i,X)\otimes\overline{\alpha(Y,t_i)}+ \overline{\alpha(Y,s_i)}\otimes \omega(t_i,X)$$
These terms cancel pairwise using Lemma \ref{(66)} and the definition of $\omega$. For example:
$$\omega(s_i,Y)\otimes\overline{\alpha(X,t_i)}\,=\,\omega(s_i,Y)\otimes\mu(Q(\alpha(X,t_i)))\,=\,
\mu(\omega(s_i,Y))\otimes \omega(X,t_i)\,=\,\overline{\alpha(s_i,Y)}\otimes \omega(X,t_i).$$

The remaining eight terms not containing $\alpha$ cancel pairwise thanks to 
the invariance property of $r$. Namely, applying the linear map 
$\omega(Y,-)\otimes{\rm id}$ to the invariance relation
$$s_i\otimes\overline{[X,t_i]}+[X,s_i]\otimes\overline{t}_i,$$
we obtain (using the anti-symmetry of $\omega$)
$$-\omega(s_i,Y)\otimes\overline{[X,t_i]}-\omega(Y,[s_i,X])\otimes\overline{t}_i.$$
This proves the claim. We go on with:
\begin{eqnarray*}
C([\overline{X},\overline{Y}])-\overline{X}\cdot C(\overline{Y})+
Y\cdot C(\overline{X})\,=\,C(\alpha(X,Y),\overline{[X,Y]})-
\overline{X}\cdot C(\overline{Y})+Y\cdot C(\overline{X})\,=\,\\
\,=\,-1_{V_2}\otimes\overline{[X,Y]}-\overline{[X,Y]}\otimes 1_{V_2}.
\end{eqnarray*}
Thus summing over all contributions proves condition {\it (c)} for 
$\xi([\overline{X},\overline{Y}])$.  

Now we show condition {\it (c)} for $\xi([\overline{X},\overline{h}])$. We start by
computing:

$$\xi_0([\overline{X},\overline{h}])\,=\,\xi_0(\overline{X\cdot h})\,=\,
\sum_i (s_i\cdot Q(\overline{X\cdot h}))\otimes \overline{t}_i+\overline{s}_i\otimes
(t_i\cdot Q(\overline{X\cdot h})),$$

$$\overline{X}\cdot\xi_0(\overline{h})\,=\,\sum_i(\overline{X}\cdot(s_i\cdot Q(h)))\otimes 
\overline{t}_i+(s_i\cdot Q(h))\otimes[\overline{X},\overline{t}_i]+
[\overline{X},\overline{s}_i]\otimes(t_i\cdot Q(h))
+\overline{s}_i\otimes(X\cdot (t_i\cdot Q(h))),$$

$$\overline{h}\cdot\xi_0(\overline{X})\,=\,\sum_i-\omega(s_i,X)\otimes(\overline{t_i\cdot h})
-(\overline{s_i\cdot h})\otimes \omega(t_i,X).$$

Observe that $X\cdot(Y\cdot Q(h))\,=\,X\cdot(Q(Y\cdot h))$; i.e. $Q$ is 
${\mathfrak g}$-equivariant when subject to a ${\mathfrak g}$-action. This 
follows from
$$\mu(Y\cdot Q(h)-Q(Y\cdot h))\,=\,Y\cdot(\mu(Q(h)))-Y\cdot h \,=\,0,$$
because $\ker[\mu)$ is ${\mathfrak g}$-invariant.  

All this is used to write

\begin{eqnarray*}
\xi_0([\overline{X},\overline{h}]-\overline{X}\cdot\xi_0(\overline{h})+
\overline{h}\cdot\xi_0(\overline{X})\,=\,   \\
\,=\,\sum_i -([X,s_i]\cdot Q(h))\otimes\overline{t}_i-\overline{s}_i\otimes([X,t_i]\cdot Q(h))
-(s_i\cdot Q(h))\otimes[\overline{X},\overline{t}_i]- \\
\,\,[ \overline{X},\overline{s}_i ] 
\otimes (t_i \cdot Q(h)) - \omega(s_i,X)\otimes(\overline{t_i\cdot h})
-(\overline{s_i\cdot h})\otimes \omega(t_i,X) \\
\,=\,\sum_i\big(-([X,s_i]\cdot Q(h))\otimes\overline{t}_i+(s_i\cdot Q(h))\otimes
\overline{[X,t_i]}\big) + \\
- \big( \overline{s}_i\otimes([X,t_i]\cdot Q(h))+
\overline{[X,s_i]}\otimes(t_i\cdot Q(h))\big)+ \\
\,-\,(s_i\cdot Q(h))\otimes\overline{\alpha(X,t_i)}
-\overline{\alpha(X,s_i)}\otimes(t_i\cdot Q(h))-
\omega(s_i,X)\otimes\overline{t_i\cdot h}-\overline{s_i\cdot h}\otimes \omega(t_i,X).
\end{eqnarray*}

In the two big parentheses, we use as before the invariance of $r$, and therefore 
they vanish. The four terms in the last line cancel pairwise, using once again Lemma 
\ref{(66)}. This completes the proof.\fin 

At last, let us address the question for which crossed modules 
$\mu:V_2\to V_3\times_{\alpha}{\mathfrak g}$ and which r-matrices $r$, the compatibility 
condition is fulfilled. Observe that one example 
(i.e. ${\mathfrak g}={\mathfrak s}{\mathfrak l}_2(\C)$)
is certainly the one given in \cite{CirMar}. This holds more generally for finite 
dimensional metric Lie algebras, thanks to the following

\begin{prop}\label{sec:constr-quasi-invar-proposition-metric-lie-compat-cond}
Let $(\mathfrak{g},\kappa)$ be a finite dimensional metric Lie algebra of dimension $n$,
$x_*:=(x_i)_{1\leq i\leq n}$ be a basis of $\mathfrak{g}$ and denote by $x^*:=(x^i)_{1\leq i\leq n}$
the $\kappa$-orthonormal basis to $x_*$, which is uniquely defined by requiring that
\[
\kappa(x_i,x^j)\,=\,\delta_{i,j}\quad,\: 1\leq i,j\leq n,
\] 
where $\delta_{i,j}$ denotes the Kronecker symbol.
Then, the $2$-tensor $r(x_*)$ defined by
\[
r(x_*):=\frac{1}{2}\sum_{i=1}^n (x_i\otimes x^i+x^i\otimes x_i) \:\in \mathfrak{g}\otimes \mathfrak{g}
\]
is an $r$-matrix for $\mathfrak{g}$ which satisfies the compatibility condition 
(\ref{compatibility}) of Cirio-Martins with respect to the Cartan cocycle 
$\kappa([-,-],-)$. Moreover, $r(x_*)$ does not depend on the choice of the basis 
$x_*$ and will be written $r_{\kappa}$ in the sequel.
\end{prop}

\pr
Let us first prove that $r(x_*)$ is an $r$-matrix. It is obviousely symmetric by 
definition, so it remains to show that it is invariant. This is the content of 
Proposition XVII.1.3 of \cite{Kas} in the case when $\kappa$ is the Killing form 
of a semi-simple Lie algebra, and the proof given there, even if it still works 
for an arbitrary metric Lie algebra, relies on the fact that the Casimir element 
metric Lie algebra is central in its universal enveloping algebra, which we don't 
need here (even if it also remains true for an arbitrary metric Lie algebra). 
Thus we propose here an alternative direct proof: First, observe that for any 
$z$ in $\mathfrak{g}$, the fact that $x_*$ is $\kappa$-orthonormal to $x^*$ implies that

\begin{equation}
z=\sum_{i=1}^n \kappa(z,x^i)x_i\label{eq:dec-orth-s}
\end{equation}
Moreover, since $\kappa$ is symmetric, we also have that

\begin{equation}
z=\sum_{i=1}^n \kappa(z,x_i)x^i\label{eq:dec-orth-t}
\end{equation}
Now let $g$ be an element of $\mathfrak{g}$. By definition of the diagonal action 
of $\mathfrak{g}$ on $\mathfrak{g}\otimes \mathfrak{g}$:
\[
g\cdot \left(\sum_{i=1}^nx_i\otimes x^i \right):=\sum_{i=1}^n \big([g,x_i]\otimes x^i+x_i\otimes [g,x^i]\big)
\]
Applying equations (\ref{eq:dec-orth-s}) with $z:=[g,x_i]$ and (\ref{eq:dec-orth-t}) 
with $z=[g,x^i]$ respectively to the first summand involving $[g,x_i]$'s  and the 
second summand involving $[g,x^i]$'s, we can rewrite the preceeding equation as 
\[
g\cdot \left(\sum_{i=1}^nx_i\otimes x^i \right)=\sum_{i=1}^n\sum_{j=1}^n 
\big(\kappa([g,x_i],x^j)x_j\otimes x^i+x_i\otimes \kappa([g,x^i],x_j)x^j\big)
\]
Using the invariance of $\kappa$ in the first summand and the fact that it 
is symmetric after having permuted $i$ with $j$ in the second, we get that
\[
g\cdot \left(\sum_{i=1}^nx_i\otimes x^i \right)=\sum_{i=1}^n\sum_{j=1}^n -
\big(\kappa(x_i,[g,x^j])x_j\otimes x^i+x_j\otimes \kappa(x_i,[g,x^j])x^i\big)\,=\,0
\]
since values of $\kappa$ can be moved from one side of the tensor product to 
the other as the tenor product is taken over $\mathbb{C}$. The same arguments apply 
to prove that
\[
g\cdot \left(\sum_{i=1}^nx^i\otimes x_i \right)=0
\]
Thus, $r(x_*)$ appears as the sum of two invariant terms i.e.
\[
g\cdot r(x_*)=g\cdot \left(\frac{1}{2}\sum_{i=1}^nx_i\otimes x^i +x^i\otimes x_i \right)=0+0=0
\]
for all $g$ in $\mathfrak{g}$, which shows that $r(x_*)$ is indeed $\mathfrak{g}$-invariant.\medskip

To establish that $r(x_*)$ fulfills the compatibility condition (\ref{compatibility}) 
with respect to the Cartan cocycle $\kappa([-,-],-)$, observe that for any crossed module \[
0\to \mathbb{C}\overset{i}{\to} \mathfrak{m}\to \mathfrak{n} \to \mathfrak{g}\to 0 
\]
representing the Cartan cocyle, the associated cochain $\Phi:\Lambda^3{\mathfrak g}\to \mathfrak{m}$ 
satisfies $\Phi(X,Y,Z)\,=\,i(\gamma(X,Y,Z))$. So, up to an identification of $\C$ with the 
subspace of $\mathfrak{m}$ generated by $1_{\mathfrak{m}}:=i(1_{\C})$, $\Phi$ is equal to the 
Cartan cocycle. Therefore, applying equation (\ref{eq:dec-orth-s}) to $z:=[X,Y]$ and the 
invariance and symmetry of $\kappa$, we get
\[
1_\mathfrak{m}\otimes [X,Y]=\sum_{i=1}^n 1_\mathfrak{m}\otimes \kappa([X,Y],x_i)x^i=\sum_{i=1}^n
\kappa([x_i,X],Y)\otimes x^i=\sum_{i=1}^n \Phi(x_i,X,Y)\otimes x^i.
\]
Similarly, using equation (\ref{eq:dec-orth-t}) instead of equation (\ref{eq:dec-orth-s}) leads to
\[
 1_{\mathfrak{m}}\otimes [X,Y] =\sum_{i=1}^n x_i\otimes \Phi(x^i,X,Y),
\]
which, averaged with the previous equation, gives
\[
1_\mathfrak{m}\otimes [X,Y]=\frac{1}{2}\sum_{i=1}^n\big(x_i\otimes \Phi(x^i,X,Y)+ x^i\otimes \Phi(x_i,X,Y)\big),
\]
 showing that $r(x_*)=\frac{1}{2}\sum_{i=1}^n\big(x_i\otimes x^i+x^i\otimes x_i\big)$ satisfies 
the left equation of the comptibility condition (\ref{compatibility}). The right 
equation can be obtained exactly in the same fashion.\medskip

The fact that $r_\kappa:=r(x_*)$ does not depend on the choice of $x_*$ is 
standard in the literature. The proof we propose here in the sake of self-containdeness 
is a straightforward adaptation to the arbitrary metric Lie algebra case of the one of 
Proposition XVII.1.2 of \cite{Kas}, in which it is shown that, a posteriori, the Casimir 
element of a semi-simple Lie algebra does not depend on the choice of the basis used to 
define it.

 If $y_*:=(y_i)_{1\leq i\leq n}$ is another basis of $\mathfrak{g}$ with orthonormal basis 
$y^*:=(y^i)_{1\leq i\leq n}$ we have, using equations (\ref{eq:dec-orth-s}) and 
(\ref{eq:dec-orth-t}) to express the $x_i$'s and the $x^i$'s respectively in the 
basis $y_*$ and $y^*$:
\begin{align*}
\sum_{i=1}^n x_i\otimes x^i=&\sum_{i=1}^n \sum_{1\leq j,k\leq n}\kappa(x_i,y^j)y_j\otimes \kappa(x^i,y_k)y^k\\
=&\sum_{1\leq j,k\leq n}\left(\sum_{i=1}^n\kappa(x_i,y^j) \kappa(x^i,y_k) \right)y_j\otimes y^k\\
\end{align*}
But for all $j$ and $k$ in $\{1,\cdots,n\}$, equation (\ref{eq:dec-orth-t}) applied to $z=y^j$ gives
\[
\sum_{i=1}^n\kappa(x_i,y^j) \kappa(x^i,y_k)=\kappa(\sum_{i=1}^n\kappa(x_i,y^j) x^i,y_k)=\kappa(y^j,y_k)=\delta_{jk}
\]
which shows that the previous equality can be rewritten as
\[
\sum_{i=1}^n x_i\otimes x^i=\sum_{1\leq j,k\leq n}\delta_{j,k} \:y_j\otimes y^k=\sum_{j=1}^n y_j\otimes y^j
\]
Symmetrizing both sides, we get 
\[
r(x_*)=r(y_*)
\]
so that the standard $r$-matrix of a metric algebra indeed does not depend on the choice 
of the basis used to define it.
\fin

\begin{cor}
Let ${\mathfrak g}$ be a semi-simple complex Lie algebra with Killing form $\kappa$.
Then the standard $r$-matrix of $\mathfrak{g}$ is the unique $r$-matrix of $\mathfrak{g}$ 
which both lifts its Casimir element and fulfills the compatibility condition 
(\ref{compatibility}) with respect to the Cartan cocycle.  
\end{cor}
\pr
By definition, the Casimir element $\Omega$ of $\mathfrak{g}$ is the image in 
$U\!\mathfrak{g}$ of its standard $r$-matrix $r_\kappa\in \mathfrak{g}\otimes\mathfrak{g}$ 
under the multiplication map $U\!\mathfrak{g}^{\otimes 2}\to U\!\mathfrak{g}$ restricted to 
$\mathfrak{g}^{\otimes 2}$, i.e.
\[
\Omega=\frac{1}{2}\sum_{i=1}^n\big(x_ix^i+x_ix^i\big) \in U\!\mathfrak{g}
\]
for any basis $(x_i)_{1\leq i\leq n}$ with $\kappa$-orthonormal basis $(y_i)_{1\leq i\leq n}$. 
Thus, $r_\kappa$ is certainly a symmetric lift of $\Omega$ to $\mathfrak{g}^{\otimes 2}$ 
and Proposition \ref{sec:constr-quasi-invar-proposition-metric-lie-compat-cond} shows 
that it is invariant and fulfills the compatibility condition (\ref{compatibility}) 
with respect to the Cartan cocycle. Since the associative algebra $U\!\mathfrak{g}$ 
is the quotient of the free associative algebra $T\!\mathfrak{g}$ by a non homogenous 
ideal, the restriction of its multiplication to $\mathfrak{g}^{\otimes 2}$ is injective, 
which proves the unicity part of the statement.
\fin

\begin{cor}
 Let ${\mathfrak g}$ be a semi-simple complex Lie algebra and let the crossed module
$\mu:V_2\to V_3\times_{\alpha}{\mathfrak g}$ be any abelian representative for the 
(cohomology class represented by) Cartan cocycle $\kappa([-,-],)$. Then the triple 
$(\bar{r}_{\kappa},\xi,c)$ defined by dots $(a)$, $(b)$ and $(c)$ of section 
\ref{construction} is a quasi-invariant tensor for $\mu$.
\end{cor}

\begin{rem}
\begin{enumerate}
\item[(a)] Observe that the short exact sequence considered earlier
$$0\to \C \to (U{\mathfrak g})^{\vee} \to (U{\mathfrak g}^+)^{\vee}\to 0$$
gives rise to an abelian representative of the crossed module corresponding to 
the Cartan cocycle of the form
\[
0\to \C\to (U{\mathfrak g})^{\vee} \to (U{\mathfrak g}^+)^{\vee}
\times_{\alpha}\mathfrak{g}\to\mathfrak{g}\to 0
\]
where $\alpha$ is the lift defined by formula (\ref{eq:formule-definition-alpha}). 
Recall from section \ref{construction} that the $\xi_0$ part of the quasi-invariant 
tensor associated to the standard $r$-matrix $r_\kappa=\sum_i s_i\otimes t_i$ of 
$\mathfrak{g}$ satisfies
\[
\xi_0(g)\,=\,\sum_i\big(\alpha(-,s_i,g)\otimes t_i+
s_i\otimes \alpha(-,t_i,g)\big)\:\:\in (U{\mathfrak g})^{\vee}\otimes 
\mathfrak{g}\:\oplus\: \mathfrak{g}\otimes (U{\mathfrak g})^{\vee}
\]
for all $g$ in $\mathfrak{g}\subset (U{\mathfrak g}^+)^{\vee}\oplus_\alpha\mathfrak{g}$. 
Thus $\xi_0(g)$ can be thought of as the (direct) sum $\xi_0^1(g)\oplus \xi_0^2(g)$ of 
two elements of $\text{\rm Hom}_\mathbb{C}(U\!\mathfrak{g},\mathfrak{g})$. Choosing 
$\mathfrak{g}:=\mathfrak{sl}_2(\mathbb{C})$ with the standard basis $\{X,Y,H\}$ 
defined above and $g:=Y$, the computation of $\alpha(XY,-,-)$ given in 
Example \ref{example2-calcul-alphaXY-X-Y-SL2} shows that
\[
\xi_0^1(Y)(XY)=\sum_i\alpha(XY,s_i,Y) t_i=\frac{1}{4}\alpha(XY,X,Y) Y=\frac{1}{3}Y
\]
and similarly
\[
\xi_0^2(Y)(XY)=\frac{1}{3}Y
\]  
where we have used that the standard $r$-matrix of $\mathfrak{sl}_2(\mathbb{C})$ 
is given by $r=\frac{1}{4}(X\otimes Y+Y\otimes X +\frac{1}{2}H\otimes H)$.
\item[(b)] Observe that the compatibility relation is also trivially satisfied 
for all abelian Lie algebras. Moreover, for abelian Lie algebras, the formula
(\ref{eq:formule-definition-alpha}) can be evaluated explicitly.  
 \end{enumerate}
\end{rem}

\end{document}